\documentclass{article}

\usepackage[microtype]{mystyle}
\usepackage{authblk}
\usepackage[margin=1in]{geometry}

\usepackage[normalem]{ulem}

\DeclareMathOperator{\linspan}{span}

\newlength\fatwidth
\title{Generating nested quadrature~rules with positive weights based on arbitrary sample sets}

\author[1,2]{L.M.M.~van~den~Bos\footnote{Corresponding author: \texttt{l.m.m.van.den.bos@cwi.nl}}}
\author[1]{B.~Sanderse}
\author[2]{W.A.A.M.~Bierbooms}
\author[2]{G.J.W.~van~Bussel}

\affil[1]{Centrum~Wiskunde~\&~Informatica, P.O.~Box 94079, 1090~GB, Amsterdam}
\affil[2]{Delft~University~of~Technology, P.O.~Box 5, 2600~AA, Delft}

\begin{document}
\maketitle

\begin{abstract}
\noindent For the purpose of uncertainty propagation a new quadrature rule technique is proposed that has positive weights, has high degree, and is constructed using only samples that describe the probability distribution of the uncertain parameters. Moreover, nodes can be added to the quadrature rule, resulting into a sequence of nested rules. The rule is constructed by iterating over the samples of the distribution and exploiting the null space of the Vandermonde-system that describes the nodes and weights, in order to select which samples will be used as nodes in the quadrature rule. The main novelty of the quadrature rule is that it can be constructed using any number of dimensions, using any basis, in any space, and using any distribution. It is demonstrated both theoretically and numerically that the rule always has positive weights and therefore has high convergence rates for sufficiently smooth functions. The convergence properties are demonstrated by approximating the integral of the Genz test functions. The applicability of the quadrature rule to complex uncertainty propagation cases is demonstrated by determining the statistics of the flow over an airfoil governed by the Euler equations, including the case of dependent uncertain input parameters. The new quadrature rule significantly outperforms classical sparse grid methods.
\end{abstract}

\begingroup
\endgroup

\section{Introduction}
The problem of uncertainty propagation is considered, where the interest is in the effect of uncertainties in model inputs on model predictions. The distribution of the quantity of interest is assessed non-intrusively, i.e.\ by means of collocation. Problems of this form occur often in engineering applications if boundary or initial conditions are not known precisely. The canonical approach is firstly to identify uncertain input parameters, secondly to define a distribution on these parameters, and finally to determine statistics of the quantity of interest~\cite{Xiu2010,LeMaitre2010,Najm2009}. These statistics are defined as integrals and various techniques exist to approximate these. However, in practice it often occurs that the distribution of the uncertain parameters is only known through a sequence or collection of samples and that the distribution is possibly correlated, e.g.\ the distribution is inferred through Bayesian analysis. The goal of this work is to construct collocation methods that are accurate for determining integrals when only samples of the distribution are known.

Several approaches exist to tackle problems of this type. In many cases the well-known and straightforward Monte Carlo approach is not applicable due to its low convergence rate of $1 / \sqrt{N}$ (with $N$ number of model evaluations) and instead collocation techniques based on polynomial approximation can be constructed to alleviate this for reasonably small dimensionality. Often these techniques are based on knowledge about the input distribution, for example its moments. A popular technique to choose evaluation nodes is the sparse grid technique~\cite{Smolyak1963,Novak1999}, which has been extended to a more general, correlated setting (mostly in a Bayesian setting, e.g.~\cite{Schillings2013,Zhang2013,Cotter2010,Franck2016}), provided that high order statistics of the distribution are known exactly. Other collocation techniques that can be applied to the setting in this work are techniques to consider the collocation problem as a minimization problem of an integration error~\cite{Sinsbeck2015,Jakeman2017}, to construct nested rules based on interpolatory Leja sequences~\cite{Leja1957,Narayan2014,Bos2018}, or to apply standard quadrature techniques after decorrelation of the distribution~\cite{Navarro2015,Feinberg2018}. All these approaches provide high order convergence, but require that the input distribution is explicitly known.

On the other hand, procedures that directly construct collocation sequences on samples without using the input distribution directly have seen an increase in popularity, possibly due to the recent growth of data sets. A recent example is the clustering approach proposed in~\cite{Eggels2018}. Another technique is based on polynomial approximation directly based on data~\cite{Oladyshkin2012} or iteratively with a focus on large data sets~\cite{Shin2017,Wu2017}. These approaches do not require stringent assumptions on the input distribution, but often do not provide high order convergence.

In this article, we propose a novel nested quadrature rule that has positive weights. There are various existing approaches to construct quadrature rules with positive weights. Examples include numerical op\-ti\-mi\-za\-tion techniques~\mbox{\cite{Keshavarzzadeh2018a,Jakeman2017,Ryu2014}}, where oftentimes the nodes and weights are determined by minimization of the quadrature rule error. A different technique that is closely related to the approach discussed in this article is subsampling~\mbox{\cite{Piazzon2017,Seshadri2017,Bos2016b,Wilson1969}}, where the quadrature rule is constructed by subsampling from a larger set of nodes. Subsampling has also been used in a randomized setting~\mbox{\cite{Wu2017}}, i.e.\ by randomly removing nodes from a large tensor grid, or to deduce a proof for Tchakaloff's theorem~\mbox{\cite{Bayer2006,Davis1967}}

The quadrature rule proposed in this work is called the \emph{implicit quadrature rule}, because it is constructed using solely samples from the distribution. The nodes of the rule form a subset of the samples and the accompanying weights are obtained by smartly exploiting the null space of the linear system governing the quadrature weights. Using a sample set limits the accuracy of the rule to the accuracy of the sample set, but an arbitrarily sized sample set can be used without additional model evaluations. The computational cost of our proposed algorithm scales (at least) linearly in the number of samples and for each sample the null space of a Vandermonde-matrix has to be determined (whose number of rows equals the number of the nodes of the quadrature rule). The main advantage of using a sample set is that the proposed quadrature rule can be applied to virtually any number of dimensions, basis, space, or distribution without affecting the computational cost of our approach. Moreover it can be extended to obtain a sequence of nested distributions, allowing for refinements that reuse existing (costly) model evaluations.

This article is set up as follows. In Section~\ref{sec:nomenclature} the nomenclature and properties of quadrature rules that are relevant for this article are discussed. In Section~\ref{sec:implicitqrule} the implicit quadrature rule is introduced and its mathematical properties are discussed. The accuracy of the quadrature rule is demonstrated by integration of the Genz test functions and by determining the statistical properties of the output of a stochastic partial differential equation modeling the flow over an airfoil. The numerical results of these test cases are discussed in Section~\ref{sec:numerics} and conclusions are drawn in Section~\ref{sec:conclusion}.

\section{Preliminaries}
\label{sec:nomenclature}
The quantity of interest is modeled as a function $u: \Omega \to \mathbb{R}$, where $\Omega$ is a domain in $\mathbb{R}^d$ (with $d = 1, 2, 3, \dots$). The parameters $x \in \Omega$ are uncertain and their distribution is characterized by an arbitrarily large set of samples, denoted by $Y_K \coloneqq \{y_0, \dots, y_K\} \subset \Omega$ (with $K \in \mathbb{N}$). In other words, the parameters $x$ have the following discrete distribution:
\begin{equation}
	\rho_K(x) = \frac{1}{K+1} \sum_{k=0}^K \delta(\| x - y_k \|),
\end{equation}
where $\delta$ denotes the usual Dirac delta function and $\| \cdot \|$ denotes any norm (the only necessary property is that $\| a \| = 0$ if and only if $a \equiv 0$). The function $u$ is not known explicitly, but can be determined for specific values of $x \in \Omega$ (e.g.\ it is the solution of a system of partial differential equations). The goal is to determine statistical moments of $u(x)$, e.g.\ to accurately determine
\begin{equation}
	\label{eq:leintegral}
	\mathcal{I}^{(K)} u \coloneqq \int_\Omega u(x) \, \rho_K(x) \dd x = \frac{1}{K+1} \sum_{k=0}^K u(y_k),
\end{equation}
where higher moments can be determined by replacing $u(x)$ with $u(x)^j$ for given $j$. Notice that if $y_k$ are samples drawn from a known (possibly continuous) distribution $\rho$, \eqref{eq:leintegral} approximates an integral weighted with this distribution, i.e.
\begin{equation}
	\label{eq:exactintegral}
	\mathcal{I}^{(K)} u = \int_\Omega u(x) \, \rho_K(x) \dd x \approx \mathcal{I} u \coloneqq \int_\Omega u(x) \, \rho(x) \dd x.
\end{equation}

We will assume throughout this work that a large number of samples can be determined fast and efficiently or is provided beforehand. There exist various methods to construct samples from well-known distributions (such as the Gaussian, Beta, and Gamma distribution)~\cite{Devroye1986}, from general distributions by means of acceptance rejection approaches, or from unscaled probability density functions by means of Markov chain Monte Carlo methods~\cite{Metropolis1953,Hastings1970}. An example of acceptance rejection sampling that we will use throughout this text to visualize our methods is depicted in Figure~\ref{fig:logos-acceptance-reject}.

\begin{figure}
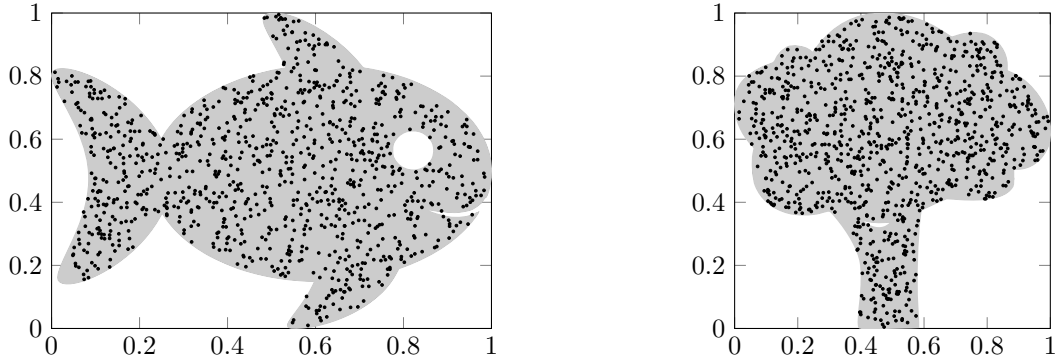

	\begin{minipage}{.5\textwidth}
		\centering
		\includepgf{.9\textwidth}{.7\textwidth}{fish-samples.tikz}
	\end{minipage}%
	\begin{minipage}{.5\textwidth}
		\centering
		\includepgf{.7\textwidth}{.7\textwidth}{tree-samples.tikz}
	\end{minipage}
	\caption{The example used throughout this work: a uniform distribution restricted to the gray sets (with 1000 samples dotted in black).}
	\label{fig:logos-acceptance-reject}
\end{figure}

If $K+1$ samples $Y_K = \{y_0, \dots, y_K\}$ are given, \eqref{eq:leintegral} could naively be evaluated by determining $u(y_k)$ for all $k$. However, it is well-known that such an approximation is very computationally costly in many practical problems. Instead we approximate the moments by means of a quadrature rule, i.e.\ the goal is to determine a finite number of nodes, denoted by the indexed set $X_N = \{x_0, \dots, x_N\} \subset \Omega$, and weights, denoted by $W_N = \{w_0, \dots, w_N\} \subset \mathbb{R}$ such that
\begin{equation}
	\mathcal{I}^{(K)} u \approx \sum_{k=0}^N u(x_k) w_k \eqqcolon \mathcal{A}^{(K)}_N u.
\end{equation}
The operator $\mathcal{A}^{(K)}_N$ is the quadrature rule operator using the nodal set $X_N$. We omit the number of samples $K$ from the notation if it is clear from the context.

Three properties are relevant in deriving quadrature rules: accuracy, positivity, and nesting. These properties are briefly discussed in the upcoming Sections~\ref{subsec:degree},~\ref{subsec:positivity},~and~\ref{subsec:nesting}. The terms \emph{nodes} and \emph{samples} are sometimes used interchangeably in a quadrature rule setting. This is not the case in this article: samples are elements from sample sets statistically describing a distribution (called $Y_K$) whereas nodes are the collocation points from a quadrature rule (called $X_N$).

\subsection{\texorpdfstring{Accuracy}{Accuracy}}
\label{subsec:degree}
We enforce that the quadrature rule is accurate on a finite-dimensional function space, denoted by $\Phi_D \coloneqq \allowbreak \linspan\{ \varphi_0, \dots, \varphi_D \}$ throughout this article. Here $\varphi_0, \dots, \varphi_D$ are basis polynomials with $\deg \varphi_j \leq \deg \varphi_k$ for $j \leq k$, such that $\Phi_D \subset \Phi_{D+1}$ for any $D$. In the univariate case, this is equivalent to enforcing that the quadrature rule has degree $D$. In the multivariate case, the quadrature rule has at least degree $Q$ if $\dim \Phi_D \geq \binom{Q+d}{d}$. The operator $\mathcal{A}^{(K)}_N$ is linear, hence if $D = N$ and $K$ is given, the weights can be determined from the nodes by solving the following linear system:
\begin{equation}
	\label{eq:linearsystem}
	\mathcal{A}^{(K)}_N \varphi_j = \mathcal{I}^{(K)} \varphi_j, \text{ for $j = 0, \dots, D$}.
\end{equation}
In the univariate case, this linear system is non-singular if all nodes are distinct. This does not hold in general in the multivariate case or if $D \neq N$.

The linear system \eqref{eq:linearsystem} will be used often in this work to ensure the accuracy of the constructed quadrature rules. The matrix of this system is called the (multivariate) Vandermonde-matrix $V_N$. If a basis $\varphi_0, \dots, \varphi_D$ is given, the system of \eqref{eq:linearsystem} can be alternatively written as
\begin{equation}
	\label{eq:vandermonde}
	V_D(X_N) \mathbf{w} \coloneqq \begin{pmatrix}
		\varphi_0(x_0) & \cdots & \varphi_0(x_N) \\
		\vdots & \ddots & \vdots \\
		\varphi_D(x_0) & \cdots & \varphi_D(x_N)
	\end{pmatrix}
	\begin{pmatrix}
		w_0 \\
		\vdots \\
		w_N
	\end{pmatrix}
	=
	\begin{pmatrix}
		\mu^{(K)}_0 \\
		\vdots \\
		\mu^{(K)}_D
	\end{pmatrix}.
\end{equation}
Here, $\mu^{(K)}_j$ are known as the (multivariate) raw moments of the samples $Y_K$, i.e.
\begin{equation}
	\label{eq:samplemoments}
	\mu^{(K)}_j \coloneqq \frac{1}{K+1} \sum_{k=0}^K \varphi_j(y_k) = \mathcal{I}^{(K)} \varphi_j.
\end{equation}

Throughout this article it is assumed that $\Phi_D$ is a polynomial space of minimal degree and that $\varphi_k$ is a monomial for each $k$. Multivariate polynomials are sorted using the graded reverse lexicographic order. All methods discussed in this article can also be applied if the polynomials are sorted differently (i.e.\ a sparse or an orthonormal basis is considered) or if the basis under consideration is not polynomial at all (e.g.\ sinusoidal). The only imposed restriction is that $\varphi_0$ is the constant function.

The matrix $V_D(X_N)$ might become ill-conditioned if it is constructed using monomials even for small $N$. Since this matrix is used to construct quadrature rules in this article, this can limit the applicability of the methods discussed here. In this article, all quadrature rules have been constructed using (products of) Legendre polynomials, which resulted in a sufficiently well-conditioned matrix for moderately large $N$ and $D$.

\subsection{Positivity, stability, and convergence}
\label{subsec:positivity}
Any constructed quadrature rule in this article has solely positive weights for two reasons: stability and convergence. We call such a quadrature rule simply a \emph{positive} quadrature rule. Both stability and convergence follow from the fact that the induced $\infty$-norm of $\mathcal{A}^{(K)}_N$ (which is the condition number of $\mathcal{A}^{(K)}_N$ as $\mu^{(K)}_0 = 1$) equals the sum of the absolute weights, i.e.
\begin{equation}
	\| \mathcal{A}^{(K)}_N \|_\infty \coloneqq \sup_{\| u \|_\infty = 1} | \mathcal{A}^{(K)}_N u | = \sum_{k=0}^N |w_k|, \text{ with } \| u \|_\infty \coloneqq \max_{x \in \Omega} |u(x)|.
\end{equation}
This norm is minimal for quadrature rules with positive weights. In these cases, we have that for all $K$:
\begin{equation}
	\label{eq:stability}
	\| \mathcal{A}^{(K)}_N \|_\infty = \sum_{k=0}^N |w_k| = \sum_{k=0}^N w_k = 1 = \mathcal{I}^{(K)} 1.
\end{equation}
If a function $u$ is perturbed by a numerical error $\vareps$, say $\tilde{u} = u + \vareps$, this does not significantly effect $\mathcal{A}^{(K)}_N u$:
\begin{equation}
	| \mathcal{A}^{(K)}_N u - \mathcal{A}^{(K)}_N \tilde{u} | \leq \| \mathcal{A}^{(K)}_N \|_\infty |u - \tilde{u}| = \vareps.
\end{equation}
This demonstrates that a quadrature rule with positive weights is numerically stable, regardless of the nodal set under consideration.

Convergence can be demonstrated similarly. This can be observed by applying the Lebesgue inequality~\cite{Brass2011}. To this end, let $q_D$ be the \emph{best approximation} of $u$ in $\Phi_D$, i.e.
\begin{equation}
	q_D = \argmin_{q \in \Phi_D} \| u - q \|_\infty.
\end{equation}
Here, we assume without loss of generality that this best approximation exists. By using $\mathcal{A}^{(K)}_N q_D = \mathcal{I}^{(K)} q_D$, the Lebesgue inequality follows:
\begin{align}
	\label{eq:lebesgue}
	| \mathcal{I}^{(K)} u - \mathcal{A}^{(K)}_N u | &= | \mathcal{I}^{(K)} u - \mathcal{I}^{(K)} q_D + \mathcal{I}^{(K)} q_D - \mathcal{A}^{(K)}_N u | \\
	&= | \mathcal{I}^{(K)} u - \mathcal{I}^{(K)} q_D + \mathcal{A}^{(K)}_N q_D - \mathcal{A}^{(K)}_N u | \\
	&\leq | \mathcal{I}^{(K)} u - \mathcal{I}^{(K)} q_D | + | \mathcal{A}^{(K)}_N q_D - \mathcal{A}^{(K)}_N u | \\
	&= | \mathcal{I}^{(K)} (u - q_D) | + | \mathcal{A}^{(K)}_N (q_D - u) | \\
	&\leq \| \mathcal{I}^{(K)} \|_\infty \| u - q_D \|_\infty + \| \mathcal{A}^{(K)}_N \|_\infty \| u - q_D \|_\infty.
\end{align}
If $w_k = | w_k |$, it holds that $\| \mathcal{A}^{(K)}_N \|_\infty = \| \mathcal{I}^{(K)} \|_\infty = 1$ (see \eqref{eq:stability}) and convergence follows readily if $\| u - q_D \|_\infty \to 0$ for $D \to \infty$, i.e.
\begin{equation}
	\label{eq:convergence}
	| \mathcal{I}^{(K)} u - \mathcal{A}^{(K)}_N u | \leq 2 \inf_{q \in \Phi_D} \| u - q \|_\infty.
\end{equation}
The rate of convergence depends on the specific characteristics of $u$: if the space $\Phi_D$ (here polynomials) is suitable for approximating $u$, the error of the quadrature rule will decay fast (e.g.\ exponentially fast if $u$ is analytic). It is well known that $u$ can be approximated well using a polynomial if among others $u$ is absolute continuous in a closed and bounded set $\Omega$, but various other results on this topic exist~\cite{Jackson1982,Watson1980,Brass2011}.

Notice that the error of the quadrature rule $\mathcal{A}^{(K)}_N u$ with respect to $\mathcal{I}^{(K)} u$ does not depend on the accuracy of the moments $\mu^{(K)}_j$, i.e.\ on whether the number of samples is large enough to resolve $\mu^{(K)}_j$ accurately. This can be seen as follows. Assume the samples $Y_K$ are drawn from a distribution $\rho: \Omega \to \mathbb{R}$ and let $\mathcal{I}$ be the integral from \eqref{eq:exactintegral} weighted with this distribution. Even though $|\mathcal{I} \varphi_j - \mathcal{I}^{(K)} \varphi_j|$ can become large for increasing $j$, the error of the quadrature rule is not necessarily large:
\begin{equation}
	| \mathcal{I} u - \mathcal{A}^{(K)}_N u | \leq \underbrace{| \mathcal{I} u - \mathcal{I}_{\phantom{N}}^{(K)} u |}_{\text{Sampling error}} + \underbrace{| \mathcal{I}_{\phantom{N}}^{(K)} u - \mathcal{A}^{(K)}_N u|}_{\text{Quadrature error}}.
\end{equation}
The error depends on two components. The sampling error describes whether the number of samples is large enough to approximate the integral of $u$ (which is independent of $\varphi_j$), whereas the quadrature error describes whether the quadrature rule is accurate (which depends on $\varphi_j$, but not through the samples, see \eqref{eq:convergence}). The quadrature error is conceptually different than the sampling error and often decreases much faster in $N$ than the sampling error does in $K$. As we assume an arbitrarily sized sequence of samples is readily available to make the sampling error sufficiently small, this article will focus on the quadrature error.

\subsection{Nesting}
\label{subsec:nesting}
Nesting means that $X_{N_1} \subset X_{N_2}$ for some $N_1 < N_2$, i.e.\ the nodes of a smaller quadrature rule are contained in a larger quadrature rule. This allows for the reuse of model evaluations if the quadrature rule is refined by considering more nodes. We will call such a quadrature rule, with a little abuse of nomenclature, a \emph{nested} quadrature rule (because strictly speaking it is a nested \emph{sequence} of quadrature rules).

A nested quadrature rule has the additional favorable property that it can be used to provide an error estimate of the approximated integral. If the quadrature rule converges to the true integral, i.e.\ $\mathcal{A}_N u \to \mathcal{I} u$, then $| \mathcal{A}_{N_1} u - \mathcal{A}_{N_2} u |$ should converge to 0:
\begin{equation}
	| \mathcal{A}_{N_1} u - \mathcal{A}_{N_2} u | \leq | \mathcal{A}_{N_1} u - \mathcal{I} u | + | \mathcal{A}_{N_2} u - \mathcal{I} u | \to 0, \text{ for $N_1, N_2 \to \infty$}.
\end{equation}
Therefore, the quantity $| \mathcal{A}_{N_1} u - \mathcal{A}_{N_2} u |$ can be used to estimate the accuracy of $A_{N_2} u$. If $X_{N_1} \subset X_{N_2}$, this error estimation can be calculated without any additional model evaluations.

\section{Implicit quadrature rule}
\label{sec:implicitqrule}
The implicit quadrature rule is a quadrature rule that is constructed using an arbitrarily sized sequence of samples. The crucial equation in the method is \eqref{eq:linearsystem}, which can be written as
\begin{equation}
	\label{eq:keyproperty}
	\sum_{k=0}^N \varphi_j(x_k) w_k = \mu^{(K)}_j, \text{ with } \mu^{(K)}_j = \frac{1}{K+1} \sum_{k=0}^K \varphi_j(y_k), \text{ for $j = 0, \dots, D$}.
\end{equation}
Given a sequence of basis functions $\varphi_0, \dots, \varphi_D$, the left hand side of this equation only depends on the quadrature nodes $X_N$ and weights $W_N$, whereas the right hand side of the equation only depends on the samples $Y_K$. The goal is to determine, based on the $K+1$ samples in the set $Y_K$, a subset of $N+1$ samples that form the nodes $X_N$ of a quadrature rule in such a way that \eqref{eq:keyproperty} is satisfied and such that the corresponding weights are positive. The existence of such a subset is motivated by the Tchakaloff bound~\cite{Davis1967}, which states that there exists a quadrature rule with positive weights with $N = D$ if $\Phi_D$ encompasses polynomials (as in this article).

\begin{figure}
	\centering
	\includegraphics{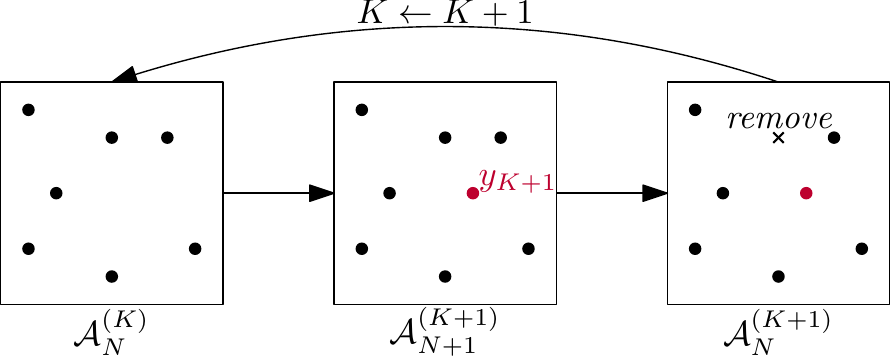}
	\caption{The implicit quadrature rule proposed in this work. Given a quadrature rule that integrates $K$ samples ($\mathcal{A}^{(K)}_N$), a node $y_{K+1}$ is added such that a rule is obtained of one more node ($\mathcal{A}^{(K+1)}_{N+1}$). Finally, one or more nodes are removed to obtain a quadrature rule of fewer nodes ($\mathcal{A}^{(K+1)}_N$).}
	\label{fig:algorithm}
\end{figure}

The approach to determine the quadrature rule is to use an iterative algorithm: starting from an initial quadrature rule, the nodes and weights are changed iteratively while new samples $y_k$ are added. Redundant nodes are removed while ensuring that the accuracy of the quadrature rule does not deteriorate. This iterative step, which is the key idea of the proposed algorithm, is sketched in Figure~\ref{fig:algorithm}. By repeatedly applying this step, a quadrature rule that validates \eqref{eq:keyproperty} is obtained.

Our algorithm is explained in the next two sections. First, in Section~\ref{subsec:basic} we propose a method for a slightly simpler problem: we fix $D$ (or $\Phi_D$) and determine at which nodes the model should be evaluated to integrate the sample moments while preserving positivity of weights. Second, in Section~\ref{subsec:extending} this method is extended to create sequences of \emph{nested} quadrature rules with increasing $D$, increasing $K$, or both. In other words, given $N$ model evaluations, we determine a subset of the samples such that \eqref{eq:keyproperty} is satisfied \emph{and} the provided $N$ model evaluations are reused.

\subsection{Fixed implicit rule}
\label{subsec:basic}
The goal is to construct a positive quadrature rule integrating all $\varphi \in \Phi_D$ exactly, where $D$ is provided a priori. The quadrature rule will consist of (at most) $D+1$ nodes in this case. Without loss of generality, it is assumed that $D < K$, i.e.\ the number of available samples is at least as large as the dimension of $\Phi_D$.

The initial step is to consider $Y_D$ and to construct the following quadrature rule for $N = D$:
\begin{align}
	X^{(D)}_N &= Y_D = \{y_0, \dots, y_D\}, \\
	W^{(D)}_N &= \{1/(D+1), \dots, 1/(D+1)\}.
\end{align}
The upper index describes the set of samples used for the construction, in this case $Y_D$, and the lower index describes the number of nodes of the quadrature rule (i.e.\ $x_0, \dots, x_N$). This initial rule simply approximates the moments by means of Monte Carlo and it is obvious that \eqref{eq:keyproperty} holds for $K = D$.

The iterative procedure works as follows. Assume $X^{(K)}_N$, $W^{(K)}_N$ form the positive quadrature rule integrating all $\varphi \in \Phi_D$ exactly. This quadrature rule has the property that $\mathcal{A}^{(K)}_N \varphi_j = \mu_j^{(K)}$ for $j = 0, \dots, D$. The goal is to construct a quadrature rule that also has this property, but with the moments $\mu_j^{(K+1)}$ as right hand side. To this end, let $y_{K+1}$ be the next sample and straightforwardly determine $X^{(K+1)}_{N+1}$ and $W^{(K+1)}_{N+1}$ as follows:
\begin{align}
	\label{eq:update1}
	X^{(K+1)}_{N+1} &= X^{(K)}_N \cup \{ y_{K+1} \}, \\
	W^{(K+1)}_{N+1} &= \left(\left( \frac{K+1}{K+2} \right) \cdot W^{(K)}_N\right) \cup \left\{ \frac{1}{K+2} \right\},
\end{align}
i.e.\ $y_{K+1}$ is ``added'' to $X^{(K)}_N$ (hence $x_{N+1} = y_{K+1}$) and the weights are changed such that the quadrature rule again integrates the sample moments. The latter can be seen as follows:
\begin{align}
	\sum_{k=0}^N \varphi_j(x_k) \frac{K+1}{K+2} w_k + \frac{1}{K+2} \varphi_j(x_{N+1}) &= \frac{K+1}{K+2} \sum_{k=0}^N \varphi_j(x_k) w_k + \frac{1}{K+2} \varphi_j(x_{N+1}) \\
	&= \frac{K+1}{K+2} \left(\frac{1}{K+1} \sum_{k=0}^K \varphi_j(y_k)\right) + \frac{1}{K+2} \varphi_j(y_{K+1}) \\
	&= \frac{1}{K+2} \sum_{k=0}^{K+1} \varphi_j(y_k) = \mu^{(K+1)}_j.
\end{align}
Here, $w_k$ are the weights from the \emph{original} quadrature rule, i.e.\ $w_k \in W^{(K)}_N$. We will use $v_k$ to denote the weights from the updated quadrature rule, i.e.\ $v_k \in W^{(K+1)}_{N+1}$.

If $W^{(K)}_N$ consists of positive weights, then so does $W^{(K+1)}_{N+1}$. The problem with this simple update is that, compared to the original nodal set, the quadrature rule now requires an additional node to integrate all $\varphi \in \Phi_D$ exactly, resulting in a total of $N+2$ nodes.

In order to construct a quadrature rule that requires only $N+1$ nodes (while preserving positive weights and integrating $\mu_j^{(K+1)}$ exactly), one node will be removed from the extended rule $X^{(K+1)}_{N+1}$, following the procedure outlined in~\cite{Bos2016b}. The procedure has an insightful geometric interpretation, as it is based on Carath\'eodory's theorem and convex cones. In this article the linear algebra interpretation is used in order to facilitate the removal of multiple nodes later in this work.

The Vandermonde-matrix of the extended quadrature rule, i.e.\ $V_D(X^{(K+1)}_{N+1})$, is as follows:
\begin{equation}
	V_D(X^{(K+1)}_{N+1}) = \begin{pmatrix}
		\varphi_0(x_0) & \dots & \varphi_0(x_N) & \varphi_0(x_{N+1}) \\
		\vdots & \ddots & \vdots & \vdots \\
		\varphi_D(x_0) & \dots & \varphi_D(x_N) & \varphi_D(x_{N+1})
	\end{pmatrix}.
\end{equation}
This is a $(D+1) \times (N+2)$-matrix (with $N = D$), so at least one non-trivial null vector $\mathbf{c} = \trans{(c_0, \dots, c_{N+1})}$ of this matrix exists, i.e.
\begin{equation}
	\begin{pmatrix}
		\varphi_0(x_0) & \dots & \varphi_0(x_N) & \varphi_0(x_{N+1}) \\
		\vdots & \ddots & \vdots & \vdots \\
		\varphi_D(x_0) & \dots & \varphi_D(x_N) & \varphi_D(x_{N+1})
	\end{pmatrix}
	\begin{pmatrix}
		c_0 \\
		\vdots \\
		c_N \\
		c_{N+1}
	\end{pmatrix}
	=
	\mathbf{0}.
\end{equation}
Any multiple of $\mathbf{c}$ is also a null vector. Hence it holds for any $\alpha \in \mathbb{R}$ that
\begin{equation}
	\begin{pmatrix}
		\varphi_0(x_0) & \dots & \varphi_0(x_N) & \varphi_0(x_{N+1}) \\
		\vdots & \ddots & \vdots & \vdots \\
		\varphi_D(x_0) & \dots & \varphi_D(x_N) & \varphi_D(x_{N+1})
	\end{pmatrix}
	\begin{pmatrix}
		\alpha c_0 \\
		\vdots \\
		\alpha c_N \\
		\alpha c_{N+1}
	\end{pmatrix}
	=
	\mathbf{0},
\end{equation}
and by combining this with \eqref{eq:keyproperty}, but now for $\mu^{(K+1)}_j$, we obtain the following:
\begin{equation}
	\begin{pmatrix}
		\varphi_0(x_0) & \dots & \varphi_0(x_N) & \varphi_0(x_{N+1}) \\
		\vdots & \ddots & \vdots & \vdots \\
		\varphi_D(x_0) & \dots & \varphi_D(x_N) & \varphi_D(x_{N+1})
	\end{pmatrix}
	\begin{pmatrix}
		v_0 - \alpha c_0 \\
		\vdots \\
		v_N - \alpha c_N \\
		v_{N+1} - \alpha c_{N+1}
	\end{pmatrix}
	=
	\begin{pmatrix}
		\mu_0^{(K+1)} \\
		\vdots \\
		\mu_D^{(K+1)} \\
	\end{pmatrix}.
\end{equation}
This equation can be interpreted as a quadrature rule depending on the free parameter $\alpha$ with nodes $X^{(K+1)}_{N+1}$ and weights $\{v_k - \alpha c_k \mid k = 0, \dots, N+1\}$. The parameter $\alpha$ can be used to remove one node from the quadrature rule, as nodes with weight equal to zero can be removed from the quadrature rule without deteriorating it. There are two options, $\alpha = \alpha_1$ or $\alpha = \alpha_2$:
\begin{align}
	\label{eq:alpha}
	\alpha_1 &= \min_k\left( \frac{v_k}{c_k} \mid c_k > 0 \right) \eqqcolon \frac{v_{k_1}}{c_{k_1}}, \\
	\alpha_2 &= \max_k\left( \frac{v_k}{c_k} \mid c_k < 0 \right) \eqqcolon \frac{v_{k_2}}{c_{k_2}}.
\end{align}
The sets $\{v_k - \alpha_1 c_k\}$ and $\{v_k - \alpha_2 c_k \}$ consist of non-negative weights and (at least) one weight equal to zero. Both $\alpha_1$ and $\alpha_2$ are well-defined, because $\mathbf{c}$ has both positive and negative elements. The latter follows from the fact that $\varphi_0$ is assumed to be a constant and that $\mathbf{c}$ is not equal to the zero vector, i.e.
\begin{equation}
	0 = \sum_{k=0}^{N+1} \varphi_0(x_k) c_k = \varphi_0 \sum_{k=0}^{N+1} c_k.
\end{equation}

The desired quadrature rule that integrates all $\varphi \in \Phi_D$ exactly and consists of $N = D$ nodes can be constructed by choosing either $i = 1$ or $i = 2$, and determining the nodes and weights as follows:
\begin{align}
	\label{eq:update2}
	X^{(K+1)}_N &= X^{(K+1)}_{N+1} \setminus \{ x_{k_i} \}, \\
	W^{(K+1)}_N &= \{v_k - \alpha_i c_k \mid k = 0, \dots, k_i-1, k_i+1, \dots, N+1 \}.
\end{align}
This rule has $N+1$ nodes and integrates the moments $\mu_j^{(K+1)}$ for $j = 0, \dots, D$ exactly. Note that, to include the case of two weights becoming zero simultaneously (e.g.\ the symmetric quadrature rules of~\cite{Bos2016b}), these sets can be implicitly defined as follows:
\begin{align}
	X^{(K+1)}_Q &= \left\{ x_k \mid x_k \in X^{(K+1)}_{N+1} \text{ and } v_k > \alpha_i c_k \right\}, \\
	W^{(K+1)}_Q &= \left\{ v_k - \alpha_i c_k \mid v_k \in W^{(K+1)}_{N+1} \text{ and } v_k > \alpha_i c_k \right\},
\end{align}
with $Q \leq N \leq D$. Without loss of generality, we assume that $Q = N$ throughout this article.

The correctness of this method follows from the fact that the first $D+1$ sample moments of the first $K$ samples are integrated exactly using the constructed quadrature rule after iteration $K$. Therefore by construction the following theorem is proved.

\begin{theorem}
	\label{thm:basiccase}
	Let $\mathcal{A}^{(K)}_N$ be a positive quadrature rule operator such that
	\begin{equation}
		\mathcal{A}^{(K)}_N \varphi_j = \mu^{(K)}_j, \text{ for $j = 0, \dots, D$},
	\end{equation}
	with $N = D$. Then after applying the procedure above, a positive quadrature rule operator $\mathcal{A}^{(K+1)}_N$ is obtained such that
	\begin{equation}
		\mathcal{A}^{(K+1)}_N \varphi_j = \mu^{(K+1)}_j, \text{ for $j = 0, \dots, D$}.
	\end{equation}
\end{theorem}

\begin{figure}[t]
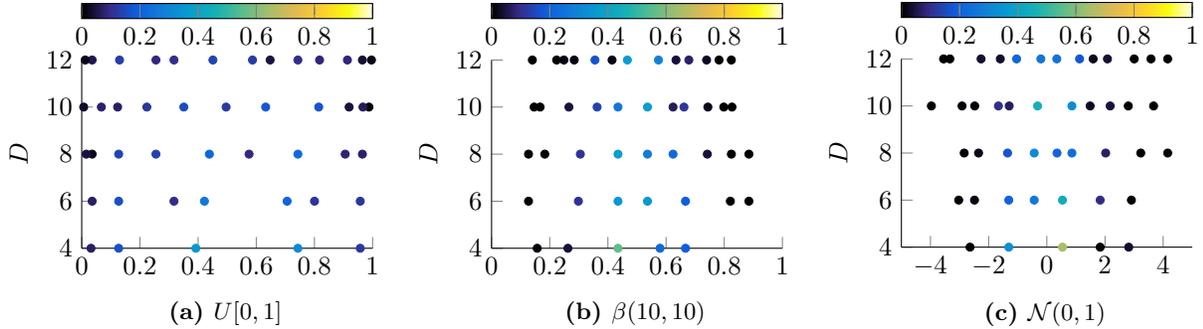

	\centering
	\begin{minipage}{.33\textwidth}
		\centering
		\includepgf{\textwidth}{.75\textwidth}{uniform.tikz}
		\subcaption{$U[0, 1]$}
	\end{minipage}%
	\begin{minipage}{.33\textwidth}
		\centering
		\includepgf{\textwidth}{.75\textwidth}{beta1010.tikz}
		\subcaption{$\beta(10, 10)$}
	\end{minipage}%
	\begin{minipage}{.33\textwidth}
		\centering
		\includepgf{\textwidth}{.75\textwidth}{normal.tikz}
		\subcaption{$\mathcal{N}(0, 1)$}
	\end{minipage}
	\caption{Examples of implicit quadrature rules for various degrees, using the same $10^5$ samples for each degree.}
	\label{fig:implicutqrule}
\end{figure}

\begin{figure}[t]
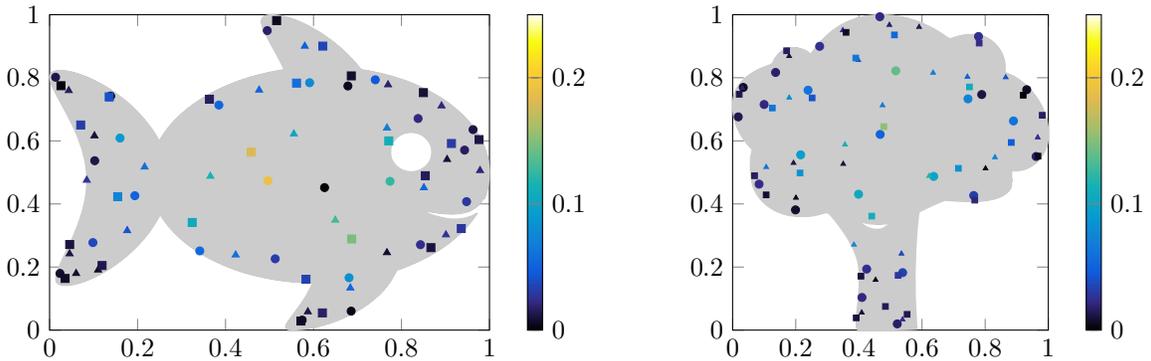

	\centering
	\begin{minipage}{.5\textwidth}
		\centering
		\includepgf{.9\textwidth}{.7\textwidth}{fish.tikz}
	\end{minipage}%
	\begin{minipage}{.5\textwidth}
		\centering
		\includepgf{.7\textwidth}{.7\textwidth}{tree.tikz}
	\end{minipage}
	\caption{Three implicit quadrature rules of 25 nodes (using different symbols) determined using the bivariate uniform distribution restricted to the gray area, using three different permutations of a set of $10^5$ samples.}
	\label{fig:logos-implicitqrule}
\end{figure}

For different sample sets, even when drawn from the same distribution, the procedure constructs different quadrature rules. If desired, this non-deterministic nature of the quadrature rule can be eradicated by using deterministic samplers, such as quasi Monte Carlo sequences~\cite{Caflisch1998}. These are not used in the quadrature rules constructed in this article, as these sequences are generally not straightforward to construct for distributions with a non-invertible cumulative distribution function. Another aspect of the algorithm that can create variation in the resulting quadrature rules, is the choice of the parameter $\alpha$. It is possible to incorporate knowledge about the integrand in the choice for $\alpha$ at each iteration, but in this article the smallest value is used, because it is assumed that we do not have a priori knowledge about the integrand.

The steps of the method are outlined in Algorithm~\ref{alg:implicitqrulemath} and examples of implicit quadrature rules obtained using sample sets drawn from well-known distributions are depicted in Figures~\ref{fig:implicutqrule}~and~\ref{fig:logos-implicitqrule}. In Figure~\ref{fig:implicutqrule}, the nodes and weights are shown for various polynomial degrees $D$, based on $K_\text{max} = 10^5$ samples drawn from several common univariate distributions. For the distributions with compact support, the nodes cluster at the boundaries of the domain. However, the nodes exhibit an irregular pattern upon increasing the degree, and determining a quadrature rule with a higher degree does not result into a nested rule (this will be addressed in Section~\ref{subsec:extending}). In the second example, nodal sets are generated in two dimensions on two different irregular domains, see Figure~\ref{fig:logos-implicitqrule}. This shows a major strength of the proposed implicit quadrature rule: it can be applied to arbitrary sample sets, including domains that are not simply connected, and positive weights are still guaranteed. Depending on the ordering of the samples in the set, different nodes and weights are obtained, indicating that the quadrature rules for these sets are not unique. It is generally not possible to obtain exactly the same quadrature rule for two permutations of the sample set, since choosing either $\alpha_1$ or $\alpha_2$ can be exploited to preserve only a single node in the rule. Theoretically this can be resolved by removing multiple nodes from the rule, as will be done in the next section, though it is often unfeasible to do so.

\begin{algorithm}[tb]
\caption{The implicit quadrature rule}
\label{alg:implicitqrulemath}
\begin{algorithmic}[1]
\Require Samples $\{y_0, \dots, y_{K_\text{max}}\}$, $\Phi_D = \linspan\{\varphi_0, \dots, \varphi_D\}$
\Ensure Positive quadrature rule $X_N = \{x_0, \dots, x_N\}$, $W_N = \{x_0, \dots, x_N\}$ with $N = D$

~

\State Initialize $X^{(D)}_N = \{y_0, \dots, y_D\}$
\State Initialize $W^{(D)}_N = \{1/(D+1), \dots, 1/(D+1)\}$

\For{$K = D, \dots, K_\text{max}-1$}
	\State \parbox{\fatwidth}{\textbf{Add node:}} $X^{(K+1)}_{N+1} \gets X^{(K)}_N \cup \{ y_{K+1} \}$
	\State \parbox{\fatwidth}{~} $W^{(K+1)}_{N+1} \gets (K+1)/(K+2) W^{(K)}_N \cup \{ 1/(K+2) \}$
	\State \parbox{\fatwidth}{\textbf{Update weights:}} Construct $V_D(X^{(K+1)}_{N+1})$
	\State \parbox{\fatwidth}{~} Determine (non-trivial) $\mathbf{c}$ such that $V_D(X^{(K+1)}_{N+1}) \mathbf{c} = \mathbf{0}$
	\State \parbox{\fatwidth}{~} $\alpha_1 \gets \min_k(v_k / c_k \mid c_k > 0)$, with $v_k \in W^{(K+1)}_{N+1}$
	\State \parbox{\fatwidth}{~} $\alpha_2 \gets \max_k(v_k / c_k \mid c_k < 0)$, with $v_k \in W^{(K+1)}_{N+1}$
	\State \parbox{\fatwidth}{\textbf{Choose:}} Either $\alpha \gets \alpha_1$ or $\alpha \gets \alpha_2$
	\State \parbox{\fatwidth}{\textbf{Remove node:}} $X^{(K+1)}_N \gets \left\{ x_k \mid x_k \in X^{(K+1)}_{N+1} \text{ and } w_k > \alpha c_k \right\}$
	\State \parbox{\fatwidth}{~} $W^{(K+1)}_N \gets \left\{ w_k - \alpha c_k \mid w_k \in W^{(K+1)}_{N+1} \text{ and } w_k > \alpha c_k \right\}$
\EndFor

\State \textbf{Return} $X^{(K_\text{max})}_N$, $W^{(K_\text{max})}_N$

\end{algorithmic}
\end{algorithm}

\subsection{Nested implicit rule}
\label{subsec:extending}
The approach of the previous section can be used to construct a quadrature rule given the number of basis vectors $D$ and a fixed number of samples $K$. For varying $D$ these quadrature rules are however not nested. In this section the algorithm is extended such that the constructed quadrature rules contain nodes that can be provided beforehand. By providing the nodes of an existing quadrature rule, a sequence of nested quadrature rules can be constructed.

The problem setting is as follows. Let $X_N$ be an indexed set of quadrature rule nodes and assume a desired number of basis vectors $D$ is specified, with $D \geq N$. The goal is to add $M$ nodes to $X_N$ in order to obtain a positive quadrature rule with nodes $X_{N+M}$ that exactly integrates all $\varphi \in \Phi_D$ exactly (so $X_N \subset X_{N+M}$). Note that in general all weights will differ, i.e.\ $W_N \not\subset W_{N+M}$. We desire to add a small number of nodes, thus $M$ to be small, but it is straightforward to observe that $M$ is bounded as follows:
\begin{equation}
	D \leq N+M \leq N+D+1.
\end{equation}
The first bound $D \leq N+M$ describes that a quadrature rule constructed with our algorithms does not integrate more basis functions exactly than its number of nodes. The second bound $N+M \leq N+D+1$ describes that it is possible to simply add a quadrature rule with $D+1$ nodes to the existing quadrature rule by setting all existing weights to 0. This is often not desired in applications, but provides a theoretical bound on the number of nodes obtained using our algorithms.

Algorithm~\ref{alg:implicitqrulemath} can be straightforwardly extended to incorporate nodes that are provided beforehand. The algorithm proceeds as usual, with the difference that nodes can only be removed if they were \emph{added} during the algorithm, but not if \emph{provided} in advance. This approach yields a sequence of nested quadrature rules, but is not optimal because it results into a quadrature rule with (possibly much) more nodes than necessary. Therefore the null space of the Vandermonde-matrix $V_D(X_N)$ is multidimensional. In such a case there are multiple nodes that can be removed together, even though removing the nodes individually yields a quadrature rule with negative weights. For example, removing two nodes from the rule yields a positive rule, but removing only one of the two yields a negative rule.

In this section the focus is therefore on the removal step of Algorithm~\ref{alg:implicitqrulemath}, which is extended to incorporate the removal of multiple nodes. By combining such an algorithm with Algorithm~\ref{alg:implicitqrulemath} the nested implicit quadrature rule is obtained.

Sequentially removing multiple nodes that result into a positive quadrature rule can result into intermediate quadrature rules with negative weights. Therefore the first step is to extend the removal procedure outlined in Section~\ref{subsec:basic} such that it supports negative weights. This is discussed in Section~\ref{subsubsec:negativeweightremoval}. The main algorithm that generalizes the approach of the basic implicit rule is presented and discussed in Section~\ref{subsubsec:recursiveremoval}. With this algorithm, the nested implicit quadrature rule follows readily, which is discussed in Section~\ref{subsubsec:nestedimplicitqrule}.

\subsubsection{Negative weight removal}
\label{subsubsec:negativeweightremoval}
The procedure from the previous section determines $\alpha_1$ and $\alpha_2$ that can be used for the removal of a node. However, the equations for $\alpha_1$ and $\alpha_2$ were derived assuming positive weights. In this section, similar equations will be derived without assuming positive weights.

Let $X_N$, $W_N$ be a quadrature rule with (possibly) negative weights. The goal is to remove one node to obtain $X_{N-1}$ and $W_{N-1}$ such that the resulting quadrature rule has positive weights and $\mathcal{A}_{N-1} \varphi_j = \mathcal{A}_N \varphi_j$ for $j = 0, \dots, N-1$. As introduced before, let $V_{N-1}(X_N)$ be the respective $N \times (N+1)$ Vandermonde-matrix and let $\mathbf{c} \in \mathbb{R}^{N+1}$ be a non-trivial null vector of that matrix. The goal is to have only positive weights, hence with the same reasoning as before, we obtain the following bound:
\begin{equation}
	w_k - \alpha c_k \geq 0, \text{ for all $k$ and a certain $\alpha$}.
\end{equation}
This translates into two cases:
\begin{equation}
	\alpha~\begin{cases}
		\geq w_k / c_k & \text{for all $k$ with $c_k < 0$}, \\
		\leq w_k / c_k & \text{for all $k$ with $c_k > 0$}.
	\end{cases}
\end{equation}
Hence the following bounds should hold for any such $\alpha$:
\begin{gather}
	\alpha_\text{min} \leq \alpha \leq \alpha_\text{max}, \text{ with } \\
	\label{eq:alphamin_alphamax}
	\alpha_\text{min} = \max_k \left( \frac{w_k}{c_k} \mid c_k < 0 \right) \eqqcolon \frac{w_{k_\text{min}}}{c_{k_\text{min}}}, \\
	\alpha_\text{max} = \min_k \left( \frac{w_k}{c_k} \mid c_k > 0 \right) \eqqcolon \frac{w_{k_\text{max}}}{c_{k_\text{max}}}.
\end{gather}
Such $\alpha$ does not necessarily exist, but if it does, both $\alpha = \alpha_\text{min}$ or $\alpha = \alpha_\text{max}$ can be used to remove either the node $x_{k_\text{min}}$ or $x_{k_\text{max}}$ from the rule (as their weight becomes 0). The case with only positive weights (which was considered in Section~\ref{subsec:basic}) fits naturally in this, with $\alpha_1 = \alpha_\text{max}$ and $\alpha_2 = \alpha_\text{min}$. If all weights are positive, it is evident that $\alpha_\text{min} < 0 < \alpha_\text{max}$.

Even a stronger, less trivial result holds: if $\alpha_\text{max} < \alpha_\text{min}$, then \emph{no} node exists that results into a positive quadrature rule after removal and if $\alpha_\text{min} < \alpha_\text{max}$, there exist \emph{exactly two} nodes such that removing one of the two results into a quadrature rule with positive weights. In other words, determining $\alpha_\text{min}$ and $\alpha_\text{max}$ as above yields all possible nodes that can be removed resulting into a positive quadrature rule. The details are discussed in the proof of the following lemma.

\begin{lemma}
	\label{lmm:all1removals}
	Let $X_N$, $W_N$ be a quadrature rule integrating all $\varphi \in \Phi_N$ exactly. The following statements are equivalent:
	\begin{enumerate}
		\item \label{item:1} $\alpha_\text{min} \leq \alpha_\text{max}$.
		\item \label{item:2} There exists an $x_{k_0} \in X_N$ such that the quadrature rule with nodes $X_N \setminus \{x_{k_0}\}$ that integrates all $\varphi \in \Phi_{N-1}$ exactly has non-negative weights.
		\item \label{item:3} Let any $x_{k_0} \in X_N$ be given such that the quadrature rule with nodes $X_N \setminus \{x_{k_0}\}$ that integrates all $\varphi \in \Phi_{N-1}$ exactly has non-negative weights. Then the weights of this rule, say $W_{N-1}$, are formed by
		\begin{equation}
			W_{N-1} = \{ w_k - \alpha c_k \mid k \neq k_0 \},
		\end{equation}
		where $c_k$ are the elements of a null vector of $V_{N-1}(X_N)$ and either $\alpha = \alpha_\text{min}$ or $\alpha = \alpha_\text{max}$.
	\end{enumerate}
\end{lemma}
\begin{proof}
	The proof consists of three parts: \ref{item:1} $\to$ \ref{item:2} $\to$ \ref{item:3} $\to$ \ref{item:1}:
	\begin{enumerate}
		\item[\hypersetup{hidelinks}$(\ref{item:1} \to \ref{item:2})$] Proving \ref{item:1} $\to$ \ref{item:2} follows immediately from the removal step outlined above (see \eqref{eq:alphamin_alphamax}).

		\item[\hypersetup{hidelinks}$(\ref{item:2} \to \ref{item:3})$] Suppose \ref{item:2} holds and let $x_{k_0}$ be given. Without loss of generality assume $k_0 = N$. Let $W_{N-1}$ be the weights of the quadrature rule nodes $X_{N-1} = X_N \setminus \{x_N\}$ and let $w^{(N)}_k \in W_N$ and $w^{(N-1)}_k \in W_{N-1}$. It holds that the nodes $X_N$ and weights $W_N$ form a quadrature rule that integrates all $\varphi \in \Phi_N$ exactly and that the nodes $X_{N-1}$ and weights $W_{N-1}$ form a quadrature rule that integrates all $\varphi \in \Phi_{N-1}$ exactly. Therefore for $j = 0, \dots, N-1$ the following holds:
		\begin{equation}
			\sum_{k=0}^N \varphi_j(x_k) w^{(N)}_k = \sum_{k=0}^{N-1} \varphi_j(x_k) w^{(N-1)}_k,
		\end{equation}
		so for these $j$ it follows that
		\begin{equation}
			\sum_{k=0}^{N-1} \varphi_j(x_k) (w^{(N)}_k - w^{(N-1)}_k) + \varphi_j(x_N) w^{(N)}_N = 0.
		\end{equation}
		Hence the vector $\mathbf{c} \in \mathbb{R}^{N+1}$ with elements $c_k = w^{(N)}_k - w^{(N-1)}_k$ (and $c_N = w^{(N)}_N$) is a null vector of $V_{N-1}(X_N)$. Then it follows that $\alpha = 1$. Without loss of generality, assume that $c_{k_0} \neq 0$.

		The remainder of this part consists of demonstrating that either $\alpha = \alpha_\text{min}$ or $\alpha = \alpha_\text{max}$. Assume $c_{k_0} > 0$ (with $k_0 = N$). It holds that $w^{(N)}_{k_0} = c_{k_0}$ and $w^{(N)}_k \geq c_k$ for all other $k$. For all $k$ with $c_k > 0$ (including $k_0$), we therefore obtain
		\begin{equation}
			\frac{w^{(N)}_k}{c_k} \geq 1.
		\end{equation}
		Equality is attained at $k = k_0$, hence $1 = \min ( w^{(N)}_k / c_k \mid c_k > 0 ) = \alpha_\text{max}$. In a similar way it can be demonstrated that if $c_{k_0} < 0$, we have $1 = \max ( w^{(N)}_k / c_k \mid c_k < 0 ) = \alpha_\text{min}$, concluding this part of the proof.

		\item[\hypersetup{hidelinks}$(\ref{item:3} \to \ref{item:1})$] Suppose \ref{item:3} holds and let the weights be given as in the lemma. Let $\alpha_\text{min}$, $\alpha_\text{max}$, $k_\text{min}$ and $k_\text{max}$ be given. By definition of $\alpha_\text{max}$, it holds that $c_{k_\text{max}} > 0$ (see Equation~\eqref{eq:alphamin_alphamax}). So if $\alpha \geq \alpha_\text{max}$, then $w_{k_\text{max}} \leq \alpha c_{k_\text{max}}$. So to have positive weights, we must have $\alpha \leq \alpha_\text{max}$.

		Similarly we have that $c_{k_\text{min}} < 0$ and therefore if $\alpha \leq \alpha_\text{min}$, it holds that $w_{k_\text{min}} \leq \alpha c_{k_\text{min}}$. So to have positive weights, we must have $\alpha \geq \alpha_\text{min}$.

		If there exists an $\alpha$ such that $\alpha_\text{min} \leq \alpha$ and $\alpha \leq \alpha_\text{max}$, it must hold that $\alpha_\text{min} \leq \alpha_\text{max}$. \qedhere
	\end{enumerate}
\end{proof}
The lemma demonstrates that $\alpha_\text{min}$ and $\alpha_\text{max}$ from \eqref{eq:alphamin_alphamax} can be used to determine whether there exists a node that yields a positive quadrature rule after removal (i.e.\ if $\alpha_\text{min} \leq \alpha_\text{max}$) and if such a node exists, either $\alpha_\text{min}$ or $\alpha_\text{max}$ can be used to determine it (by determining $k_0$ as in the proof). If $\alpha_\text{max} > \alpha_\text{min}$, no such node exists, but this is not an issue, since the algorithm to construct quadrature rules discussed in this article does by construction not end up in this case.

\subsubsection{Removal of multiple nodes}
\label{subsubsec:recursiveremoval}
Let $X_N$ and $W_N$ form a positive quadrature rule. In this section, the goal is to determine all subsets of $M$ nodes that have one specific property in common: removing those $M$ nodes results in a positive quadrature rule of $N+1-M$ nodes that exactly integrates all $\varphi \in \Phi_{N-M}$. We call a subset with this property an $M$-removal. Hence in the previous section a procedure has been presented to determine all 1-removals.

Lemma~\ref{lmm:all1removals} is the main ingredient for deriving all $M$-removals. The idea boils down to the following. Let an $M$-removal be given, say $(q_1, \dots, q_M) \subset X_N$. If the first $M-1$ nodes from this $M$-removal are removed, the $M$th node $q_M$ can be determined straightforwardly using $\alpha_\text{min}$ or $\alpha_\text{max}$ from \eqref{eq:alphamin_alphamax}. There are two possible values of $\alpha$ (namely either $\alpha_\text{min}$ or $\alpha_\text{max}$), hence there exists a second node, say $\hat{q}_M$, such that $(q_1, \dots, q_{M-1}, \hat{q}_M)$ is also an $M$-removal. The order in which the nodes are removed is irrelevant, so each node $q_k$ can be replaced in this way by a different node $\hat{q}_k$ resulting in a valid $M$-removal, i.e.\ a set of $M$ nodes that can be removed while preserving positive weights and obtaining a quadrature rule that exactly integrates all $\varphi \in \Phi_{N-M}$.

We denote the procedure of obtaining a different $M$-removal from an existing one by the operator $F: {[X_N]}^M \to {[X_N]}^M$, where ${[X_N]}^M$ denotes the set of all $M$-subsets of $X_N$. If $(q_1, \dots, q_M)$ is an $M$-removal, applying $F$ yields the $M$-removal $(q_1, \dots, q_{M-1}, \hat{q}_M)$. Such an operator is well-defined, since Lemma~\ref{lmm:all1removals} prescribes that there exist exactly two $M$-removals whose first $M-1$ elements equal $q_1, \dots, q_{M-1}$. Notice that the operator $F$, which depends on the nodes and weights of the quadrature rule, can be computed by determining $\alpha_\textrm{min}$ and $\alpha_\textrm{max}$ from Lemma~\ref{lmm:all1removals} after removal of $q_1, \dots, q_{M-1}$.

By permuting the $M$-removal before applying $F$, one $M$-removal yields (up to a permutation at most) $M$ other $M$-removals. These $M$-removals can be considered in a similar fashion and recursively more $M$-removals can be determined. This procedure yields all $M$-removals, which is demonstrated in the following lemma.

\begin{figure}
	\centering
	\includegraphics{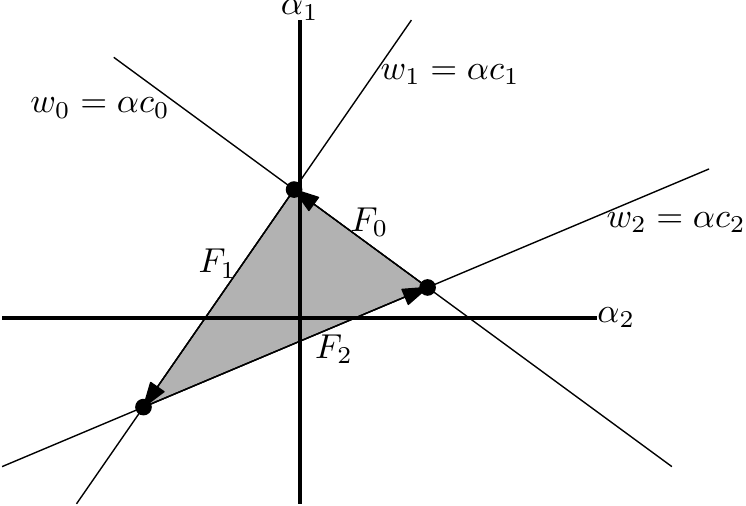}
	\caption{Graphical sketch of the simplex describing the removal of two nodes from a quadrature rule of three nodes. The gray area describes the simplex where all values of $(\alpha_1, \alpha_2)$ yield positive weights. The operator $F_k$ (see proof of Lemma~\ref{lmm:allMremovals}) can be used to traverse the boundary of the simplex.}
	\label{fig:M-removal}
\end{figure}
\begin{lemma}
	\label{lmm:allMremovals}
	Let $(q_1, \dots, q_M)$ and $(s_1, \dots, s_M)$ be any two different $M$-removals of the positive quadrature rule $X_N$, $W_N$. Let the operator $F: {[X_N]}^M \to {[X_N]}^M$, as described in the text, be such that
	\begin{equation}
		F(q_1, \dots, q_{M-1}, q_M) = (q_1, \dots, q_{M-1}, \hat{q}_M),
	\end{equation}
	for a given $M$-removal $(q_1, \dots, q_M)$, i.e.\ it replaces $q_M$ by $\hat{q}_M$ such that $(q_1, \dots, q_{M-1}, \hat{q}_M)$ is an $M$-removal. Then there exists a finite number of permutations $\sigma_1, \dots, \sigma_n$ such that
	\begin{equation}
		(\sigma_1 \circ F \circ \sigma_2 \circ F \circ \cdots \circ F \circ \sigma_{n-1} \circ F \circ \sigma_n) (q_1, \dots, q_M) = (s_1, \dots, s_M).
	\end{equation}
\end{lemma}
\begin{proof}
	Let $F_k: {[X_N]}^M \to {[X_N]}^M$ be the operator that firstly permutes $q_k$ to the end of the $M$-removal and secondly applies $F$, i.e.
	\begin{equation}
		F_k(q_1, \dots, q_M) = F(q_1, \dots, q_{k-1}, q_{k+1}, \dots, q_M, q_k) = (q_1, \dots, q_{k-1}, q_{k+1}, \dots, q_M, \hat{q}_k).
	\end{equation}
	Notice that $F_k = F \circ \pi_k$, where $\pi_k$ denotes the permutation that appends $q_k$. Hence if there exist $k_1, \dots, k_n$ such that $F_{k_1} \circ \cdots \circ F_{k_n} (q_1, \dots, q_M)$ equals $(s_1, \dots, s_M)$ up to a permutation, the proof is done.

	Consider $W_N = \{w_0, \dots, w_N\}$ and let $V_{N-M}(X_N)$ be the Vandermonde-matrix with respect to this quadrature rule. This is an $(N-M+1) \times (N+1)$-matrix, so there exist $M$ linearly independent null vectors $\mathbf{c}^{1}, \dots, \mathbf{c}^{M} \in \mathbb{R}^{N+1}$. We use the following notation for the vector $\mathbf{c}^k$:
	\begin{equation}
		\mathbf{c}^k = \trans{(c_0^k, \dots, c_N^k)}.
	\end{equation}
	Let $\alphab = (\alpha_1, \dots, \alpha_M)$ be an $M$-tuple and consider the following set:
	\begin{equation}
		S = \left\{ \alphab = (\alpha_1, \dots, \alpha_M) \mathrel{\Big|} \forall k = 0, \dots, N : w_k - \sum_{j=1}^M \alpha_j c_k^j \geq 0 \right\}.
	\end{equation}

	The set $S$ is a closed simplex, as it is formed by a finite number of linear inequalities. Moreover it is non-empty, as $(0, \dots, 0) \in S$ (this follows from the fact that $w_k \geq 0$ for all $k$). 
	
	The boundary of $S$, i.e.\ $\partial S$, is of special interest. If $(\alpha_1, \dots, \alpha_M) \in \partial S$, it holds that $w_k - \sum_{k=1}^M \alpha_k c_k = 0$ for at least one $k$. At vertices of the simplex, the highest number of weights, namely $M$, is $0$. The operator $F_k$ can be used to traverse the vertices of the simplex. The simplex $S$ for determining a 2-removal for a quadrature rule of 3 nodes is sketched in Figure~\ref{fig:M-removal}.
	
	Consider an $M$-removal $\mathbf{q} = (q_1, \dots, q_M)$. There is exactly one $M$-tuple $\boldsymbol\alpha = (\alpha_1, \dots, \alpha_M)$ resulting into the removal of these nodes. These $\alpha$'s coincide with a vertex of simplex $S$. Applying $F_k$ to $\mathbf{q}$ yields a different $M$-tuple. These two $M$-tuples are connected through an edge of the simplex. Due to Lemma~\ref{lmm:all1removals} all $M$-tuples that are connected to $\boldsymbol\alpha$ through an edge can be found. Therefore, for a given vertex, the operator $F_k$ yields all connected vertices and can be used to traverse the boundary of the simplex. This concludes the proof.
\end{proof}
The statement of the lemma is constructive: given a single $M$-removal, all $M$-removals can be found. By repetitively constructing a 1-removal using the methods from Section~\ref{subsubsec:negativeweightremoval}, an initial $M$-removal can be obtained (which is a vertex of the simplex $S$ discussed in the proof). Lemma~\ref{lmm:allMremovals} assures that any other $M$-removal can be reached from this removal. An outline of this procedure can be found in Algorithm~\ref{alg:massiveremoval}. The computational cost of calculating the null vectors can be alleviated by decomposing $V_{N-M}(X_{N-M})$ (e.g.\ using an LU or QR decomposition) once and computing the null vectors of $V_{N-M}(X_{N-M+1})$ in the loop by reusing this decomposition.

The time complexity of this algorithm is $\mathcal{O}(Z \log Z + Z (N-M)^3)$, where $Z$ is the number of $M$-removals. Here, the term $Z \log Z$ originates from storing all visited $M$-removals efficiently using a binary search tree (which results into $Z$ lookups that scale with $\log Z$) and the term $Z (N-M)^3$ is obtained by factorizing $V_{N-M}(X_{N-M})$ and repeatedly computing the null vector of $V_{N-M}(X_{N-M+1})$ using this factorization. Algorithm~\ref{alg:massiveremoval} always terminates, as the number of subsets of $M$ nodes is strictly bounded. Combining this with the proof of Lemma~\ref{lmm:allMremovals} proves the following theorem.

\begin{theorem}
	On termination, Algorithm~\ref{alg:massiveremoval} returns all $M$-removals of the positive quadrature rule $X_N$, $W_N$.
\end{theorem}

\begin{algorithm}[t]
\caption{Removing multiple nodes}
\label{alg:massiveremoval}
\begin{algorithmic}[1]
\Require Positive quadrature rule $X_N$, $W_N$, integer $M$ with $1 \leq M < N+1$
\Ensure All $M$-removals of $X_N$, $W_N$

~

\State Construct $V_{N-M}(X_N)$
\State Determine $M$ independent null vectors $\mathbf{c}^k$ of $V_{N-M}(X_N)$
\State Construct an $M$-removal, say $\mathbf{q} \gets (q_1, \dots, q_M) \subset X_N$ (e.g.\ by repeatedly using Lemma~\ref{lmm:all1removals})
\State $I \gets \{ \mathbf{q} \}$, the set containing all \emph{queued} removals
\State $R \gets \emptyset$, the set containing all \emph{processed} removals
\While {$I \neq \emptyset$}
	\State Get the first removal from $I$, say $\mathbf{q} \gets (q_1, \dots, q_M) \subset X_N$
	\State Remove $\mathbf{q}$ from $I$, i.e.\ $I \gets I \setminus \{ \mathbf{q} \}$.
	\For{$i = 1, \dots, M$}
		\State Construct $X_{N-M+1}$ and $W_{N-M+1}$ by removing $(q_1, \dots, q_{i-1}, q_{i+1}, \dots, q_M)$
		\State Determine $\mathbf{c}$ such that $V_{N-M}(X_{N-M+1}) \mathbf{c} = \mathbf{0}$, determine $\alpha_\text{min}$, $\alpha_\text{max}$, $k_\text{min}$, $k_\text{max}$ from \eqref{eq:alphamin_alphamax}
		\If{$q_i = x_{k_\text{max}}$}
			\State $\hat{q}_i \gets x_{k_\text{min}}$
		\Else
			\State $\hat{q}_i \gets x_{k_\text{max}}$
		\EndIf
		\State $\hat{\mathbf{q}} \gets (q_1, \dots, q_{i-1}, \hat{q}_i, q_{i+1}, \dots, q_M)$, which is an $M$-removal
		\If{$\hat{\mathbf{q}} \notin I$ and $\hat{\mathbf{q}} \notin R$}
			\Comment \textit{NB: this means we have not visited vertex $\hat{\mathbf{q}}$ yet.}
			\State Add $\hat{\mathbf{q}}$ to $I$, i.e.\ $I \gets I \cup \{ \hat{\mathbf{q}} \}$
		\EndIf
	\EndFor
	\State $R \gets R \cup \{ \mathbf{q} \}$
\EndWhile
\State \textbf{Return} $R$

\end{algorithmic}
\end{algorithm}

Theoretically, Algorithm~\ref{alg:massiveremoval} can be used to determine \emph{all} quadrature rules $\mathcal{A}^{(K)}_N$ with $\mathcal{A}^{(K)}_N \varphi = \mathcal{I}^{(K)} \varphi$ for all $\varphi \in \Phi_D$. All these rules are obtained by computing all $M$-removals with $M = K-N$ of the quadrature rule $X_K = Y_K$ with $w_k = 1/(K+1)$ for all $k = 0, \dots, K$. However, in practice this is intractable, as the number of $M$-removals grows rapidly in $M$ and $K-N$ is typically a large quantity.

\subsubsection{The nested implicit quadrature rule}
\label{subsubsec:nestedimplicitqrule}
\begin{figure}[t]
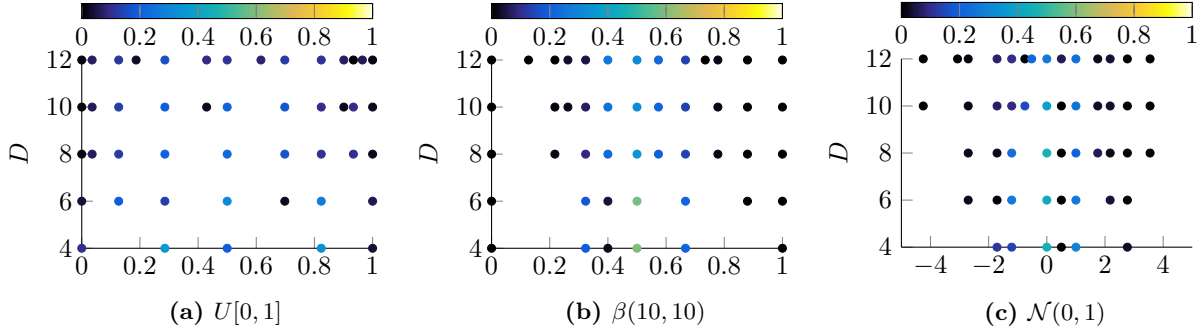

	\centering
	\begin{minipage}{.33\textwidth}
		\centering
		\includepgf{\textwidth}{.75\textwidth}{uniform-ex.tikz}
		\subcaption{$U[0, 1]$}
	\end{minipage}%
	\begin{minipage}{.33\textwidth}
		\centering
		\includepgf{\textwidth}{.75\textwidth}{beta1010-ex.tikz}
		\subcaption{$\beta(10, 10)$}
	\end{minipage}%
	\begin{minipage}{.33\textwidth}
		\centering
		\includepgf{\textwidth}{.75\textwidth}{normal-ex.tikz}
		\subcaption{$\mathcal{N}(0, 1)$}
	\end{minipage}
	\caption{Examples of nested nodal sets constructed with $10^5$ samples. The initial nodes are in all three cases $[0, 1/2, 1]$.}
	\label{fig:extendedimplicitqrule}
\end{figure}

\begin{figure}[t]
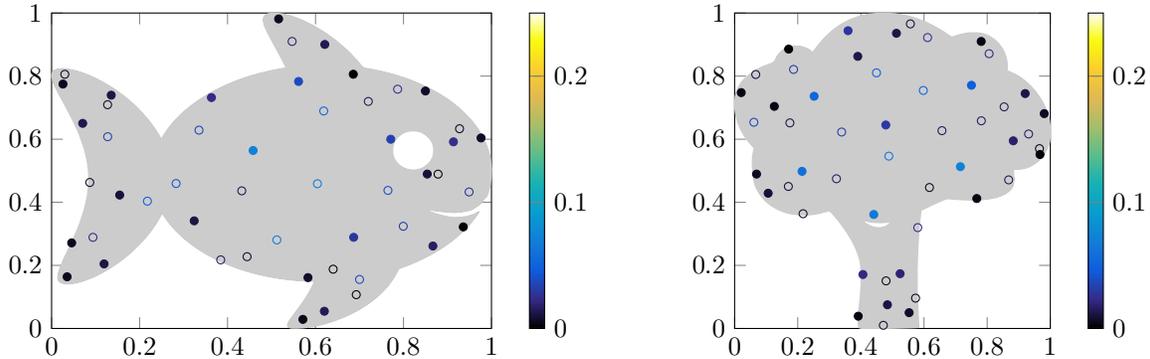

	\centering
	\begin{minipage}{.5\textwidth}
		\centering
		\includepgf{.9\textwidth}{.7\textwidth}{fish-ex.tikz}
	\end{minipage}%
	\begin{minipage}{.5\textwidth}
		\centering
		\includepgf{.7\textwidth}{.7\textwidth}{tree-ex.tikz}
	\end{minipage}
	\caption{The implicit quadrature rule of 25 nodes (closed circles) and 50 nodes (open and closed circles) respectively determined using the uniform distribution restricted to the gray area, using 100\,000 samples. The colors refer to the weights of the largest quadrature rule.}
	\label{fig:logos-extendedimplicitqrule}
\end{figure}
In this section the key algorithm of this paper is presented, namely the nested implicit quadrature rule for arbitrary sample sets. It is constructed by combining the algorithms from the previous sections. Given a quadrature rule, two different refinements (or a combination of both) are considered. Firstly, the number of samples $K$ can be increased to obtain a more accurate estimate of $\mu^{(K)}_j$. Secondly, $D$ can be increased to obtain a more accurate quadrature rule.

The set-up is similar to the one used so far, i.e.\ let $\{y_k\}$ be a sequence of samples, $X^{(K)}_N$ be a set of nodes, $W^{(K)}_N$ be a set of non-negative weights and assume the following holds for a certain $D$:
\begin{equation}
	\sum_{k=0}^N \varphi_j(x_k) w_k = \mu^{(K)}_j, \text{ for $j = 0, \dots, D$, $x_k \in X^{(K)}_N$, and $w_k \in W^{(K)}_N$}.
\end{equation}

Let $D_+$ and $K_+$ be the desired number of basis vectors and (possibly larger) number of samples respectively and assume $D_+ \geq D$. The goal is to determine $X^{(K_+)}_{N+M}$ and $W^{(K_+)}_{N+M}$ such that $W^{(K_+)}_{N+M}$ is non-negative, $X^{(K)}_N \subset X^{(K_+)}_{N+M}$, and
\begin{equation}
	\sum_{k=0}^{N+M} \varphi_j(x_k) w_k = \mu^{(K_+)}_j, \text{ for $j = 0, \dots, D_+$, $x_k \in X^{(K_+)}_{N+M}$, and $w_k \in W^{(K_+)}_{N+M}$}.
\end{equation}
In other words, we consider $K_+$ samples $Y_{K_+}$ and want to determine a positive quadrature rule that integrates all $\varphi \in \Phi_{D_+}$ exactly by adding $M$ nodes to $X_N$ (with $M$ minimal).

The iterative procedure is similar to Algorithm~\ref{alg:implicitqrulemath} and consists of 4 steps: (i) determine or obtain the next sample $y_{K+1}$, (ii) update the nodes and weights according to \eqref{eq:update1}, (iii) determine all possible removals (see Algorithm~\ref{alg:massiveremoval}), and finally (iv) remove nodes such that the obtained quadrature rule is as small as possible. The last step consists of finding the $M$-removal such that $X^{(K_+)}_{N+M} \setminus X^{(K)}_N$ (i.e.\ the set of new nodes) becomes as small as possible. If a node $x_k \in X^{(K)}_N$ is part of the $M$-removal, its weight is simply set to 0 (this is not problematic: the weights change again in subsequent iterations).

The initialization of the algorithm depends on whether {more basis vectors are considered}, a larger number of samples is considered, or the set of samples is changed (e.g.\ the sequence of samples is redrawn):
\begin{enumerate}
	\item If $D_+ = D$ and $K_+ > K$, the procedure is a continuation of the original Algorithm~\ref{alg:implicitqrulemath} and no initialization is necessary.
	\item If $D_+ > D$, we need to reiterate over all samples to determine $\mu_j^{(K)}$ for $j > D$. The algorithm can be initialized using $X^{(K)}_N$ as nodes, using all weights equal to $1/(N+1)$, and using the samples $Y_{K_+} \setminus X^{(K)}_D$.
	\item If the sequence of samples is regenerated from the underlying distribution, then in general $X^{(K)}_N \not\subset Y_{K_+}$. Therefore, the algorithm is initialized with $X^{(K+1)}_N \cup \{ y_0 \}$ and $W^{(K+1)}_N = \{0, \dots, 0, 1\}$.
\end{enumerate}

The outline of this algorithm is provided in Algorithm~\ref{alg:extendimplicitqrulemath}, which is a straightforward extension of Algorithm~\ref{alg:implicitqrulemath} with additional bookkeeping to incorporate the removal of multiple nodes. Some examples of nested sequences are gathered in Figure~\ref{fig:extendedimplicitqrule}~and~\ref{fig:logos-extendedimplicitqrule}. In the first figure all three nodal sequences are initialized with the nodes $0$, $1/2$, and $1$. If these nodes are used to construct conventional interpolatory quadrature rules, then the quadrature rule is positive if the uniform or Beta distribution is used, but has negative negative weights if the normal distribution is considered (the weights are $3$, $-4$, and $2$ respectively). However, the proposed algorithm incorporates these nodes without difficulty in subsequent quadrature rules, resulting into positive weights. Note that the quadrature rules of polynomial degree 4 have 6 nodes in case of the Beta and normal distribution. The subsequent quadrature rules of the normal distribution have higher number of nodes than the degree, which is due to the ``bad'' initial set of nodes.

A two-dimensional example is presented in Figure~\ref{fig:logos-extendedimplicitqrule}. Here, the initial quadrature rule is depicted with closed circles and its extension thereof with open circles.

\begin{algorithm}[t]
\caption{The nested implicit quadrature rule}
\label{alg:extendimplicitqrulemath}
\begin{algorithmic}[1]
\Require Samples $\{y_0, \dots, y_{K_\text{max}}\}$, quadrature nodes $X_N$, $\Phi_D = \linspan\{\varphi_0, \dots, \varphi_D\}$
\Ensure Positive quadrature rule $X_{N+M}$, $W_{N+M}$

~

\State Initialize $X^{(N)}_N$ and $W^{(N)}_N$, e.g. $X^{(N)}_N \gets X_N$, $W^{(N)}_N \gets \{ 1/(N+1), \dots, 1/(N+1) \}$
\State $M \gets 0$

\For{$K = N, \dots, K_\text{max}$}
	\State \parbox{\fatwidth}{\textbf{Add node:}} $X^{(K+1)}_{N+M+1} \gets X^{(K)}_{N+M} \cup \{y_{K-N}\}$
	\State \parbox{\fatwidth}{~} $W^{(K+1)}_{N+M+1} \gets (K+1) / (K+2) W^{(K)}_{N+M} \cup \{ 1/(K+2) \}$
	\State \parbox{\fatwidth}{\textbf{Update weights:}} Construct $V_D(X_{N+M+1}^{(K+1)})$
	\State \parbox{\fatwidth}{~} Determine null vectors $\mathbf{c}^1, \dots, \mathbf{c}^M$ and determine all $M$-removals
	\State \parbox{\fatwidth}{\textbf{Choose:}} Let $\mathbf{q} = (q_1, \dots, q_M)$ be an $M$-removal

	\Comment \textit{NB: we choose the one that makes $X^{(K+1)}_{N+M+1} \setminus X_N$ the smallest.}

	\State \parbox{\fatwidth}{\textbf{Remove node:}} Let $(\alpha_1, \dots, \alpha_n)$ such that $\hat{\mathbf{w}} \coloneqq \mathbf{w} - \sum_{j=1}^M \alpha_j c^j_k = 0$ for all $\hat{w}_j$ with $x_j \in \mathbf{q}$
	\State \parbox{\fatwidth}{~} $\hat{M} \gets \#\left\{x_k \in X^{(K+1)}_{N+M+1} \mid x_k \notin X_N \text{ and } w_k - \sum_{j=1}^M \alpha_j c^j_k = 0\right\}$
	\State \parbox{\fatwidth}{~} $X^{(K+1)}_{N+\hat{M}+1} \gets \left\{ x_k \in X^{(K+1)}_{N+M+1} \mid w_k - \sum_{j=1}^M \alpha_j c^j_k > 0 \text{ or } x_k \in X_N \right \}$
	\State \parbox{\fatwidth}{~} $W^{(K+1)}_{N+\hat{M}+1} \gets \left\{ w_k - \sum_{j=1}^M \alpha_j c^j_k \mid x_k \in X^{(K+1)}_{N+\hat{M}+1}\right\}$
	\State \parbox{\fatwidth}{~} $M \gets \hat{M}$
\EndFor

\State \textbf{Return $X^{(K_\text{max})}_{N+M}$, $W^{(K_\text{max})}_{N+M}$}

\end{algorithmic}
\end{algorithm}

Again it holds that different sample sets produce different quadrature rules. Similarly as in Section~\ref{subsec:basic}, this can be eradicated by using deterministic samplers. An additional degree of freedom arises when choosing the $M$-removal, as there might be several $M$-removals that remove the largest number of nodes from $X^{(K_+)}_{N+M} \setminus X^{(K)}_N$. In the quadrature rules constructed in this work, we select one randomly.

The large advantage of the nested implicit quadrature rule is that it is dimension agnostic, basis agnostic, space agnostic, and distribution agnostic, which are properties it carries over from the basic implicit quadrature rule. Virtually any space and any distribution can be used, as long as the distribution has finite moments and a set of samples can be generated, can be determined, or is available.

\section{Numerical examples}
\label{sec:numerics}
Two different types of test cases are employed to demonstrate the discussed properties of our proposed quadrature rule, in particular the independence from the underlying distribution.

The first class of cases exists of explicitly known test functions and distributions to assess the accuracy of the quadrature rule for integration purposes. To this end, the Genz integration test functions~\cite{Genz1984} are employed and a comparison is made with a Monte Carlo approach. Moreover, it is instructive to see how the convergence compares with that of a sparse grid, although the comparison is strictly incorrect, as a Smolyak sparse grid converges to the true integral.

Secondly a partial differential equation (PDE) with random coefficients is considered, where the goal is to infer statistical moments about the solution of a PDE with random boundary conditions. The equations under consideration are the inviscid Euler equations modeling the flow around an airfoil, where the inflow parameters and the shape of the airfoil are assumed to be uncertain.

The Genz integration test functions are studied in Section~\ref{subsec:genz}. The Euler equations are considered in Section~\ref{subsec:spde}. 

\subsection{Genz test functions}
\label{subsec:genz}
The Genz integration test functions~\cite{Genz1984} are a set of test function to assess the accuracy of numerical integration routines, see Table~\ref{tbl:genz}. Each test function has a certain attribute that is challenging for most numerical integration routines. The exact value of the integral of any of these test functions on the unit hypercube can be determined exactly. In this section, these functions are used to test the implicit quadrature rule. The goal is to assess the absolute integration error for increasing number of nodes in a 5-dimensional setting, i.e.\ to assess:
\begin{equation}
	e_N = \left|\mathcal{A}^{(K_\text{max})}_N u - \frac{1}{K_\text{max}+1} \sum_{k=0}^{K_\text{max}} u(y_k)\right|,
\end{equation}
for samples $y_0, \dots, y_{K_\text{max}}$ and various increasing $N$. The number of samples is chosen such that the quadrature error dominates and the sampling error $| \mathcal{I}^{(K_\text{max})} u - \mathcal{I} u |$ is small. We compare the approximation with that of a Monte Carlo approach, where we assess the following error:
\begin{equation}
	e_N = \left|\frac{1}{N+1} \sum_{k=0}^N u(y_k) - \frac{1}{K_\text{max}+1} \sum_{k=0}^{K_\text{max}} u(y_k)\right|,
\end{equation}
i.e.\ it is considered as a quadrature rule with nodes $\{y_k\}$ and weights $1 / (N+1)$. If the underlying distribution is tensorized, we also study the Smolyak sparse grid, which is constructed using exponentially growing Clenshaw--Curtis quadrature rules in conjunction with the combination rule~\cite{Novak1999}. The sparse grid converges to the true value of the integral and therefore we use the true value of the integral to assess its convergence, even though the comparison is not completely fair in this case.

\begin{table}
	\centering
	\caption{The test functions from Genz~\citep{Genz1984}. All $d$-variate functions depend on the $d$-element vectors $\mathbf{a}$ and $\mathbf{b}$. The vector $\mathbf{b}$ is an offset parameter to shift the function. The vector $\mathbf{a}$ describes the degree to which the family attribute is present.}
	\begin{tabular}{l l}
		\textbf{Integrand Family} & \textbf{Attribute} \\
		\hline
		\hline
		$u_1(x) = \cos\left(2\pi b_1 + \sum_{i=1}^d a_i x_i\right)$ & Oscillatory \\
		$u_2(x) = \prod_{i=1}^d \left(a_i^{-2} + (x_i - b_i)^2\right)^{-1}$ & Product Peak \\
		$u_3(x) = \left(1 + \sum_{i=1}^d a_i x_i\right)^{-(d+1)}$ & Corner Peak \\
		$u_4(x) = \exp\left(- \sum_{i=1}^d a_i^2 (x_i - b_i)^2 \right)$ & Gaussian \\
		$u_5(x) = \exp\left(- \sum_{i=1}^d a_i |x_i - b_i|\right)$ & $C_0$ function \\
		$u_6(x) = \begin{cases}
			0 &\text{if $x_1 > b_1$ or $x_2 > b_2$} \\
			\exp\left(\sum_{i=1}^d a_i x_i\right) &\text{otherwise}
		\end{cases}$ & Discontinuous
	\end{tabular}
	\label{tbl:genz}
\end{table}

The numerical experiment is repeated twice for two different input distributions. Firstly, the uniform distribution is used to be able to compare the methodology with conventional quadrature rule methods. Secondly, a highly correlated multivariate distribution (inspired by Rosenbrock function) is used to demonstrate the independence of the convergence rate from the input distribution.

To obtain meaningful results, the offset and shape parameters $\mathbf{a}$ and $\mathbf{b}$ of the Genz functions are chosen randomly and the numerical experiment is repeated 50 times. The obtained 50 absolute integration errors are averaged. The vector $\mathbf{a}$ is obtained by firstly sampling uniformly from ${[0,1]}^5$ and secondly scaling $\mathbf{a}$ such that $\| \mathbf{a} \|_2 = 5/2$. The vector $\mathbf{b}$ is uniformly distributed in ${[0, 1]}^5$ without further scaling, as it is an offset parameter.

The implicit quadrature rule is generated with $K_\text{max} = 10^4$ samples drawn randomly from the two input distributions respectively and the Monte Carlo approximation is determined using a subset of these samples, such that both the Monte Carlo approximation and the implicit quadrature rule converge to the same result. The initial quadrature rule of one single node is determined randomly and the rule is extended by applying Algorithm~\ref{alg:extendimplicitqrulemath}. Each extension is such that $D$ doubles, up to $D = 2^{10} = 1024$, but we emphasize that any granularity can be used here. Recall that $\Phi_D = \linspan\{\varphi_0, \dots, \varphi_D\}$ where $\varphi_j$ are $d$-variate polynomials sorted graded reverse lexicographically. Hence each extension integrates a larger number of polynomials exactly. For sake of completeness, a comparison is made with a non nested implicit quadrature rule, which is regenerated for each $D$ by means of Algorithm~\ref{alg:implicitqrulemath}.

\begin{figure}
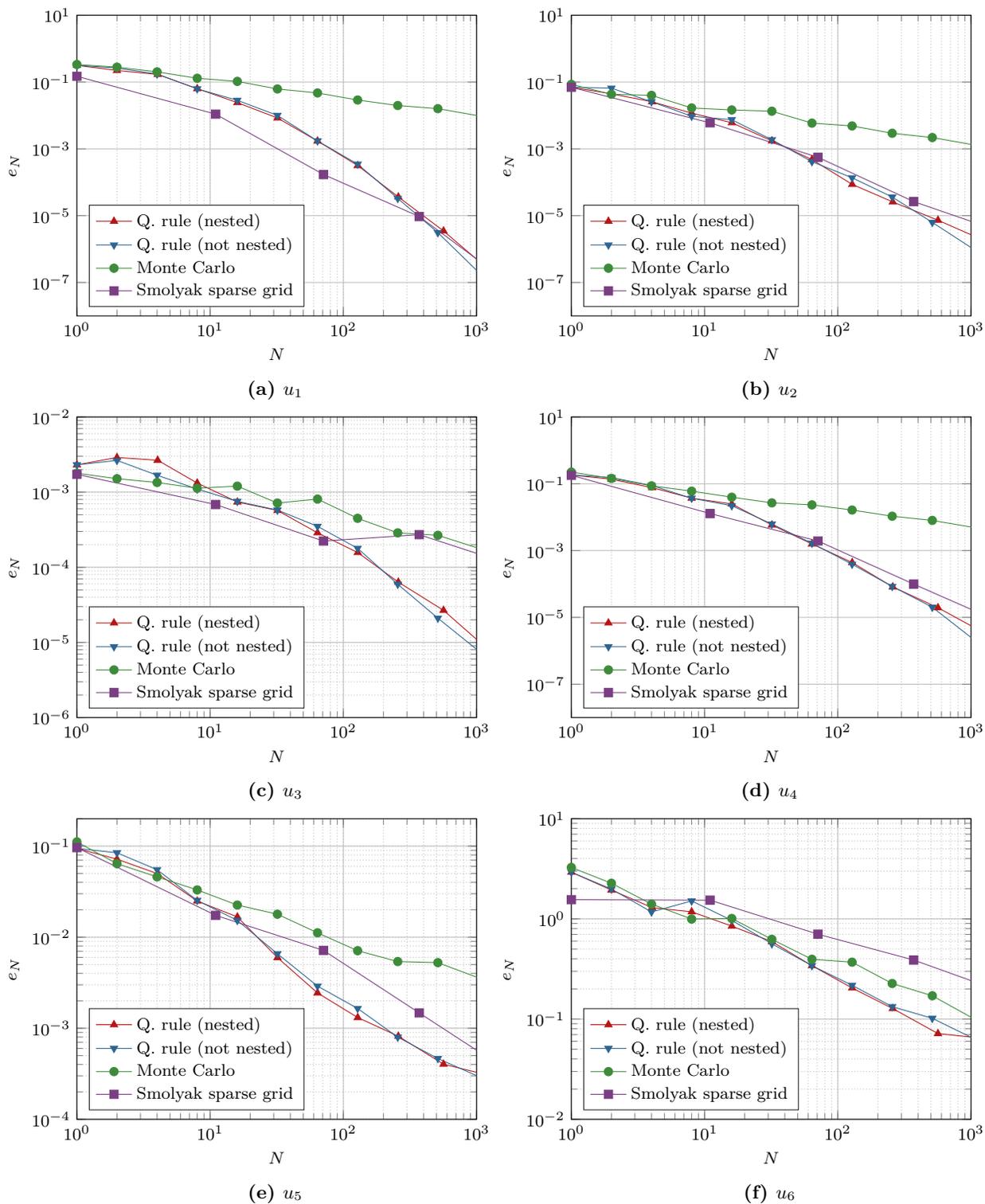

	\centering

	\begin{minipage}[t]{.5\textwidth}
		\centering
		\footnotesize
		\includepgf{\textwidth}{.8\textwidth}{genz-1.tikz}
		\subcaption{$u_1$}
	\end{minipage}%
	\begin{minipage}[t]{.5\textwidth}
		\centering
		\footnotesize
		\includepgf{\textwidth}{0.8\textwidth}{genz-2.tikz}
		\subcaption{$u_2$}
	\end{minipage}

	\begin{minipage}[t]{.5\textwidth}
		\centering
		\footnotesize
		\includepgf{\textwidth}{0.8\textwidth}{genz-3.tikz}
		\subcaption{$u_3$}
	\end{minipage}%
	\begin{minipage}[t]{.5\textwidth}
		\centering
		\footnotesize
		\includepgf{\textwidth}{0.8\textwidth}{genz-4.tikz}
		\subcaption{$u_4$}
	\end{minipage}

	\begin{minipage}[t]{.5\textwidth}
		\centering
		\footnotesize
		\includepgf{\textwidth}{0.8\textwidth}{genz-5.tikz}
		\subcaption{$u_5$}
	\end{minipage}%
	\begin{minipage}[t]{.5\textwidth}
		\centering
		\footnotesize
		\includepgf{\textwidth}{0.8\textwidth}{genz-6.tikz}
		\subcaption{$u_6$}
	\end{minipage}%

	\caption{Convergence of the absolute integration error for Genz test functions using the nested and non nested implicit quadrature rule, Monte Carlo sampling, and the Smolyak sparse grid using the uniform distribution.}
	\label{fig:convergence-uniform}
\end{figure}

\subsubsection{Uniform distribution}
The multivariate uniform distribution in ${[0,1]}^d$ (with $d = 5$ in this case) can be constructed by means of a tensor product of multiple univariate uniform distributions. It is therefore possible to approximate the integral using the well-known Smolyak sparse grid. The results of the four integration routines under consideration (Monte Carlo, nested and non nested implicit quadrature rule, and Smolyak sparse grid) are depicted in Figure~\ref{fig:convergence-uniform}. Here, $N$ denotes the number of nodes of the quadrature rules and the Smolyak sparse grid is refined by increasing the sparse grid level equally in all dimensions.

The accuracy of a quadrature rule is highly dependent on the analyticity and smoothness of the integrand. Globally analytic functions can be approximated well using polynomials, i.e.\ $\inf_{q \in \Phi_D} \| q - u \|_\infty$ decays fast. This property is reflected in the results.

The first four Genz functions (i.e.\ $u_1$, $u_2$, $u_3$, and $u_4$) are smooth and therefore the most suitable for integration by means of a quadrature rule. The best convergence is observed for the oscillatory, product peak, and Gaussian function, which are analytic. The corner peak is analytic, but has very slowly decaying derivatives, such that the quadrature rule approximation only converges exponentially fast for very large number of nodes (which are not considered here).

The continuous (but not differentiable) $C_0$ function follows a similar reasoning. It is not globally analytic, hence no exponential convergence is obtained. The Smolyak quadrature rule has a slightly larger error in this case compared to the implicit quadrature rule (arguably due to its negative weights), even though it seems that the rate of convergence is similar.

Integrating the discontinuous function by means of a positive quadrature rule does not yield any improvement over Monte Carlo sampling. The Smolyak sparse grid performs worse in this case due to its negative weights and usage of the Clenshaw--Curtis quadrature rule (which is not suitable for integration of discontinuous functions).

\begin{figure}
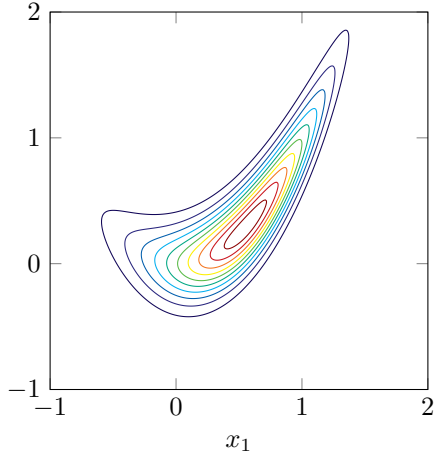

	\centering
	\includepgf{.4\textwidth}{.4\textwidth}{rosenbrock.tikz}
	\caption{The bivariate Rosenbrock distribution.}
	\label{fig:rosenbrock}
\end{figure}

\begin{figure}
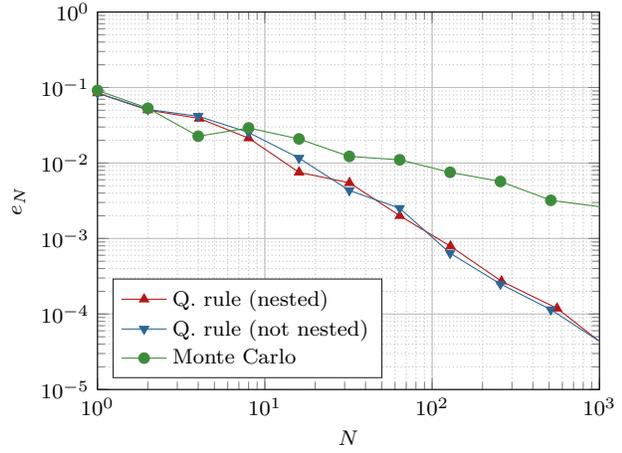
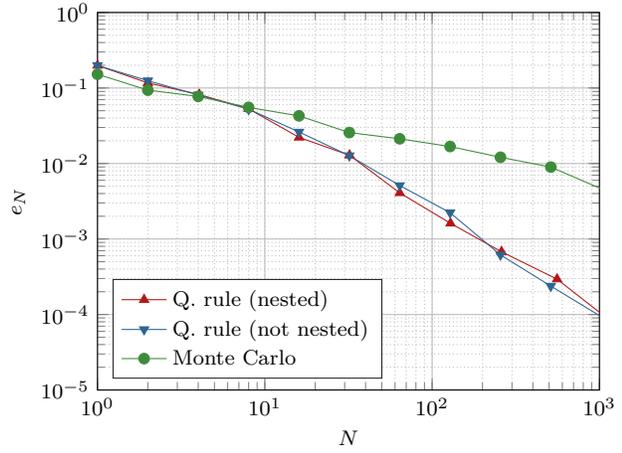
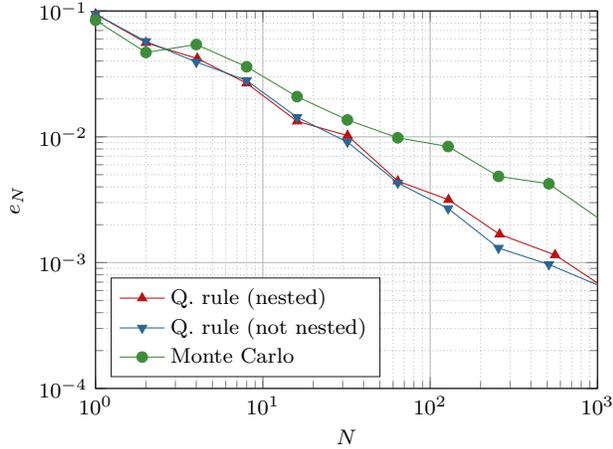
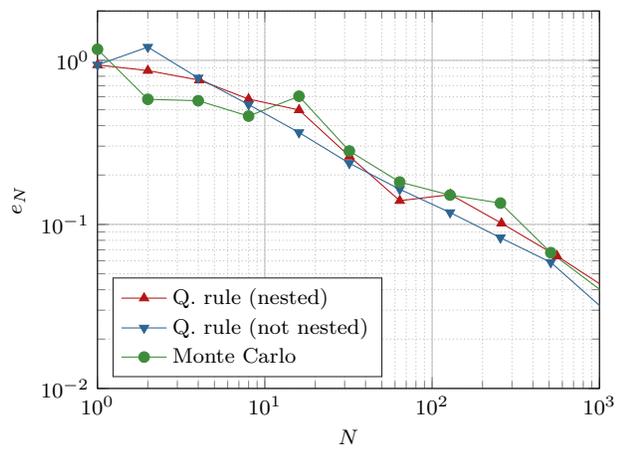

	\centering

	\begin{minipage}[t]{.5\textwidth}
		\centering
		\footnotesize
		\includepgf{\textwidth}{.8\textwidth}{genz-rb-1.tikz}
		\subcaption{$u_1$}
	\end{minipage}%
	\begin{minipage}[t]{.5\textwidth}
		\centering
		\footnotesize
		\includepgf{\textwidth}{0.8\textwidth}{genz-rb-2.tikz}
		\subcaption{$u_2$}
	\end{minipage}

	\begin{minipage}[b]{.5\textwidth}
		\parbox[c][0.7\textwidth][c]{\textwidth}{\centering (integral diverges)}
		\subcaption{$u_3$}
	\end{minipage}%
	\begin{minipage}[b]{.5\textwidth}
		\centering
		\footnotesize
		\includepgf{\textwidth}{0.8\textwidth}{genz-rb-4.tikz}
		\subcaption{$u_4$}
	\end{minipage}

	\begin{minipage}[t]{.5\textwidth}
		\centering
		\footnotesize
		\includepgf{\textwidth}{0.8\textwidth}{genz-rb-5.tikz}
		\subcaption{$u_5$}
	\end{minipage}%
	\begin{minipage}[t]{.5\textwidth}
		\centering
		\footnotesize
		\includepgf{\textwidth}{0.8\textwidth}{genz-rb-6.tikz}
		\subcaption{$u_6$}
	\end{minipage}%

	\caption{Convergence of the absolute integration error for Genz test function using the nested and non nested implicit quadrature rule and Monte Carlo sampling using the Rosenbrock distribution.}
	\label{fig:convergence-rosenbrock}
\end{figure}

\subsubsection{Rosenbrock distribution}
The large advantage of the implicit quadrature rule is that it can be constructed using any arbitrary set of samples. In order to assess this applicability to general distributions, the following distribution (which we will call the Rosenbrock distribution) is considered:
\begin{equation}
	\rho: \mathbb{R}^d \to \mathbb{R}, \text{ defined by } \rho(x) \propto \exp \left[ -f(x) \right] \pi(x),
\end{equation}
with $\pi$ the PDF of the multivariate standard Gaussian distribution and $f$ (a variant of) the multivariate Rosenbrock function:
\begin{equation}
	f(x) = f(x_1, \dots, x_d) = \sum_{k=1}^{d-1} \left[b \, (x_{i+1} - x_i^2)^2 + (a - x_i)^2\right], \text{ with $a = 1$ and $b = 10$}.
\end{equation}
The distribution $\rho$ for $d = 2$ is depicted in Figure~\ref{fig:rosenbrock}. This distribution is not optimal for integration by means of a sparse grid as it is not tensorized. Integration by means of a sparse grid converges prohibitively slow, even if the quadrature rules used for the construction are based on the marginals of the distribution. Therefore these results are omitted.

The \emph{exact} integral over the corner peak function $u_3$ diverges in this case, so approximating such an integral will result into a diverging quadrature rule. The results of the other functions are gathered in Figure~\ref{fig:convergence-rosenbrock}.

Similarly to the uniform case, the properties of the functions are reflected in the convergence rates of the approximations. The integrals of the smooth functions converge fast with a high rate, the convergence of the $C_0$ function is smaller, and the convergence of the discontinuous function is comparable to that of Monte Carlo. This result is significant, as it demonstrates that the convergence rate of $\mathcal{A}^{(K)}_N$ to the sampling-based integral $\mathcal{I}^{(K)}$ for $N \to K$ shows no significant dependence on the sample set used to construct the rules.

\subsection{Airfoil flow using Euler equations}
\label{subsec:spde}
In this section the flow over an airfoil is considered with uncertain geometry and inflow conditions. The quantity of interest is the pressure coefficient of the airfoil. The problem is five dimensional: two parameters model uncertain environmental conditions and three parameters model the uncertain geometry of the airfoil. The geometry is described by the 4-digit NACA profile and the equations governing the flow are the inviscid Euler equations. Problems of this type are well-known in the framework of uncertainty propagation~\cite{Liu2017,Witteveen2009a,Loeven2008} and allow to demonstrate the applicability of the proposed quadrature rule to a complex uncertainty propagation test case in conjunction with a complex underlying distribution.

\begin{table}
	\caption{Uncertain parameters of the airfoil test case.}
	\label{tbl:uncertainparameters}

	\centering
	\begin{tabular}{rll}
		& \textbf{Parameter} & \textbf{Distribution} \\
		\hline
		\hline
		$\alpha$ & Angle of attack & Uniform in $[0^\circ, 5^\circ]$ \\
		$M$ & Mach number & Beta$(4, 4)$ distributed in $[0.4, 0.6]$ \\
		&& \\
		$t$ & Maximum thickness of the airfoil & Beta$(4, 4)$ distributed in $[0.11, 0.13]$ \\
		$m$ & Maximum camber & Beta$(2, \cdot)$ distributed in $[0.02, 0.03]$ with mean: \\
		& & $\bar{m} = 1/4 t^2 + 0.02$ \\
		$p$ & Location of maximum camber & Uniform distributed in $[0.3, p_\text{max}]$ with: \\
		& & $p_\text{max} = \frac{1}{2} \left(2t + 16m\right)^5 + 0.3$
	\end{tabular}
\end{table}

The five uncertain parameters are summarized in Table~\ref{tbl:uncertainparameters}. The angle of attack and Mach number are distributed independently from the other parameters and describe uncertain inflow conditions. The remaining three parameters define the 4-digit NACA airfoil~\cite{Jacobs1935}. A NACA airfoil can be generated directly from these parameters and the mean of these parameters is approximately a NACA2312 airfoil. In this work the NACA airfoil with closed tip is considered by correction of the last parameter.

The compressible Euler equations are numerically solved using the finite volume solver \texttt{SU2}~\cite{Economon2016,Palacios2014}. The mesh is generated using \texttt{gmsh}~\cite{Geuzaine2009}. The implicit quadrature rules are determined by means of $K_\text{max} = 10^6$ randomly drawn samples from the distributions described in Table~\ref{tbl:uncertainparameters}, that are also being reused for consecutive refinements.

The quantity of interest in this test case is the pressure coefficient on the surface of the airfoil $C_p(x)$, i.e.\ the scaled pressure such that zero pressure equals a non obstructed flow:
\begin{equation*}
	C_p(x) = \frac{p(x) - p_\infty}{\frac{1}{2} \rho_\infty V_\infty^2}.
\end{equation*}
Here, $p(x)$ is the pressure at location $x$, $p_\infty$ is the freestream pressure (i.e.\ the pressure on the boundary in this case), $\rho_\infty$ is the freestream density of air, and $V_\infty$ is the freestream velocity of the fluid. The quadrature rule is applied piecewise to this quantity, where all pressure realizations are piecewise linearly interpolated onto the same mesh. Accuracy is measured by using the lift coefficient, which is the dimensionless coefficient relating the lift generated by the airfoil with the farfield fluid density and velocity. It follows naturally by integrating the pressure coefficient over the airfoil surface. In Figure~\ref{subfig:cp-x} we used the quadrature rule $\mathcal{A}_N u$ with $u = C_p$ and in Figure~\ref{subfig:cl-convergence} we used the quadrature rule with $u = C_l$.

The number of nodes of the quadrature rule is doubled until the difference between two consecutive quadrature rule estimations of the mean of the pressure coefficient is smaller than $10^{-2}$, which is the case at $512$ nodes. The results are summarized in Figure~\ref{fig:cp-x+convergence} and high order convergence is clearly visible. Both the mean and standard deviation show convergence with a larger rate than that of Monte Carlo (i.e.\ larger than $1/2$). We want to emphasize the importance of positive weights for this engineering test case, as it assures the estimation of the variance is non-negative even in the presence of high non-linearities.

The moments of the pressure coefficient around the airfoil are depicted in Figure~\ref{fig:cp-moments}, where the airfoil geometry that is plotted is the overlap of all airfoils (such that all depicted flow locations are in the flow for all quadrature rule nodes).

\begin{figure}
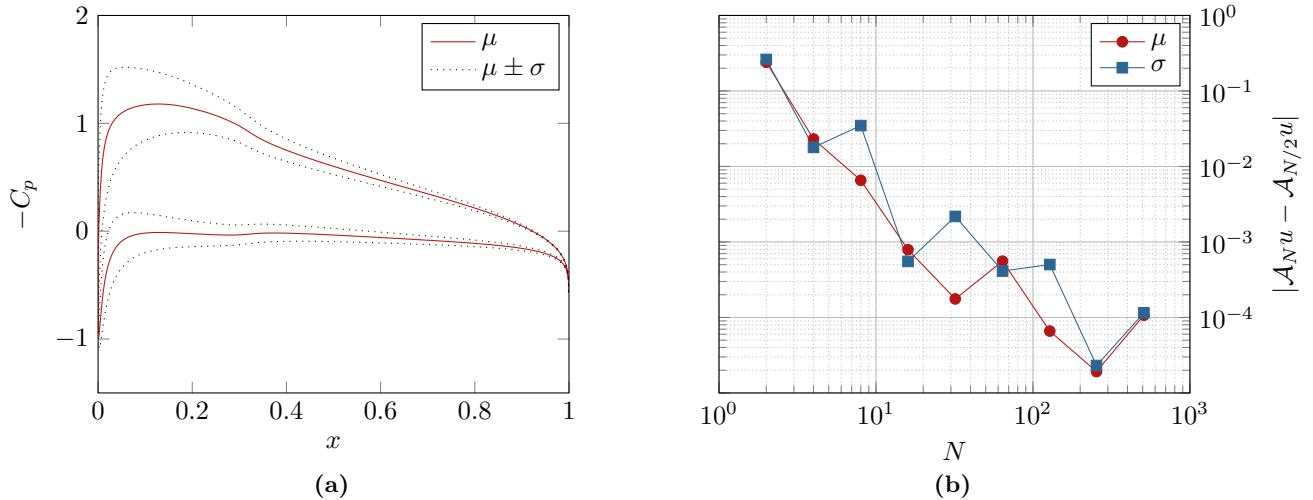

	\begin{minipage}[t]{.5\textwidth}
		\centering
		\includepgf{.95\textwidth}{.8\textwidth}{cp-x.tikz}
	\end{minipage}%
	\begin{minipage}[t]{.5\textwidth}
		\centering
		\includepgf{.95\textwidth}{.8\textwidth}{cl-convergence.tikz}
	\end{minipage}

	\begin{minipage}[t]{.5\textwidth}
		\centering
		\subcaption{}
		\label{subfig:cp-x}
	\end{minipage}%
	\begin{minipage}[t]{.5\textwidth}
		\centering
		\subcaption{}
		\label{subfig:cl-convergence}
	\end{minipage}

	\caption{\textit{Left:}~mean $\mu$ and standard deviation $\sigma$ of the pressure coefficient ($C_p$) of the airfoil determined using the finest quadrature rule. \textit{Right:}~Convergence of the mean $\mu$ and standard deviation $\sigma$ of the lift coefficient ($C_l$) by calculating consecutive differences.}
	\label{fig:cp-x+convergence}
\end{figure}

\begin{figure}
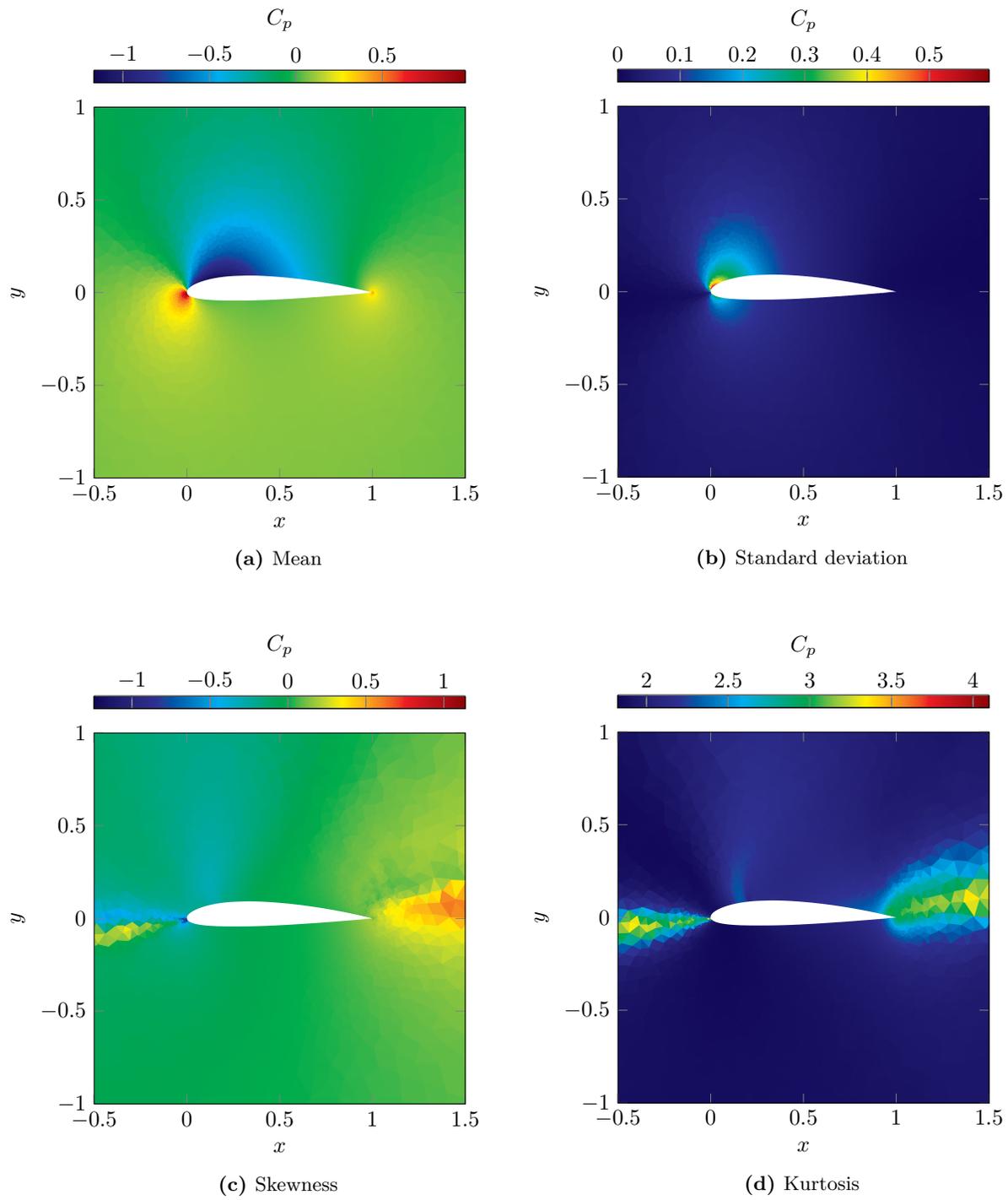

	\begin{minipage}{.5\textwidth}
		\centering
		\includepgf{.9\textwidth}{.9\textwidth}{cp-mean.tikz}
		\subcaption{Mean}
	\end{minipage}%
	\begin{minipage}{.5\textwidth}
		\centering
		\includepgf{.9\textwidth}{.9\textwidth}{cp-stddev.tikz}
		\subcaption{Standard deviation}
	\end{minipage}

	\vspace*{2\baselineskip}

	\begin{minipage}{.5\textwidth}
		\centering
		\includepgf{.9\textwidth}{.9\textwidth}{cp-skew.tikz}
		\subcaption{Skewness}
	\end{minipage}%
	\begin{minipage}{.5\textwidth}
		\centering
		\includepgf{.9\textwidth}{.9\textwidth}{cp-kurt.tikz}
		\subcaption{Kurtosis}
	\end{minipage}

	\caption{The first four centralized moments of the pressure coefficient around the airfoil.}
	\label{fig:cp-moments}
\end{figure}

The largest uncertainty of the flow is near the leading edge of the airfoil. This is in contrast to higher Mach number flows, for which it has been observed that the region of largest uncertainty occurs near the shock wave~\cite{Bos2018,Witteveen2009a}. The skewness of the pressure coefficient shows that its distribution is slightly skewed near the stagnation point on the leading edge of the airfoil. The skewness changes sign at the mean angle of attack. Moreover in the wake of the airfoil the distribution has positive skewness, which means that outliers of the distribution will more likely be larger than smaller compared to the average pressure coefficient. The kurtosis demonstrates that in many regions of the flow the distribution has much less mass in the tails than a Gaussian (and therefore more unlikely produces outliers). However, again near the leading edge and trailing edge the distribution differs and the tails of the distributions are significantly more influential. We cannot conclude that these locations are the regions of highest uncertainty, as the standard deviation (which is the scaling factor of both the skewness and the kurtosis) is very small in these regions. It is merely a sign that the uncertain behavior of the flow cannot be fully captured by a Gaussian distribution.

\section{Conclusion}
\label{sec:conclusion}
In this article, a novel nested quadrature rule is proposed which is constructed by solely using samples from a distribution. As the weights of the rule are positive, high order convergence is obtained for sufficiently smooth functions. The algorithm to construct the quadrature rule ensures positive weights, high degree, and nesting regardless of the sample set. The quadrature rules are very suitable for the purpose of non-intrusive uncertainty propagation, because positive weights ensure numerical stability and nesting allows for refinements that reuse computationally expensive model evaluations.

The results from integrating Genz test functions demonstrate that the convergence rate of the quadrature rule is similar to that of the Smolyak sparse grid approach, if the underlying sample distribution is uncorrelated and defined on a hypercube. The real advantage of the proposed quadrature rule appears when this is not the case: for correlated distribution on non hypercube domains our method still converges at a rate similar to the uncorrelated case, while a sparse grid quadrature rule hardly converges at all. Similar to the existing quadrature rules, the convergence depends on the specific properties of the integrand, in particular on its smoothness.

To demonstrate the applicability to practical test cases, the quadrature rule is used to determine the statistical moments of an airfoil flow with both independent and dependent input distributions. The results demonstrate the advantages of the quadrature rule: nesting can be used for easy refinements, positive weights ensure stability and positive approximations of positive quantities (such as the variance), dependency is naturally taken into account, and the accuracy of the rule yields high convergence rates.

The proposed algorithms provide a framework for the construction of quadrature rules that shows much potential for further extensions. For example, tailoring the basis to the integrand can yield an adaptive quadrature rule without deteriorating the accuracy of the rule as a whole. As the rule is solely based on sample sets, no stringent assumptions are necessary to apply the quadrature rule in such a different setting.

\section*{Acknowledgments}
This research is part of the Dutch EUROS program, which is supported by NWO domain Applied and Engineering Sciences and partly funded by the Dutch Ministry of Economic Affairs.

\small
\bibliographystyle{plainnatnourl}
\bibliography{article-cited}

\end{document}